\documentclass[11pt]{article}
\usepackage{amsmath,amssymb,amsthm,amsfonts,amstext,amsbsy,amscd}
\usepackage{graphics}
\usepackage{graphicx}

\RequirePackage[colorlinks,citecolor=blue,urlcolor=blue]{hyperref}

\RequirePackage{hypernat}

\def\<{\langle}
\def\>{\rangle}

\def\Chi{\raise .3ex
\hbox{\large $\chi$}}

\def\({\Bigl (}
\def\){\Bigr )}
\newcommand{\be}{\begin{equation}}
\newcommand{\ee}{\end{equation}}
\newcommand{\bea}{$$ \begin{array}{lll}}
\newcommand{\eea}{\end{array} $$}
\newcommand{\bi}{\begin{itemize}}
\newcommand{\ei}{\end{itemize}}

\numberwithin{equation}{section}
\newtheorem{satz}{Satz}[section]

\newtheorem{theorem}[satz]{Theorem}

\newtheorem{lemma}[satz]{Lemma}
\newtheorem{assumption}[satz]{Assumption}

\newtheorem{definition}[satz]{Definition}

\newtheorem{remark}[satz]{Remark}

\DeclareMathOperator{\E}{{\mathbb E}}

\DeclareMathOperator{\R}{{\mathbb R}}

\DeclareMathOperator{\Z}{{\mathbb Z}}
\DeclareMathOperator{\N}{{\mathbb N}}

\DeclareMathOperator{\PP}{{\mathbb P}}
\DeclareMathOperator{\QQ}{{\mathbb Q}}

\renewcommand{\phi}{\varphi}
\renewcommand{\cdot}{{\scriptstyle \bullet} }

\begin{document}
\title{Statistical inference across time scales}
\author{C\'eline Duval\footnote{GIS-CREST and CNRS-UMR 8050, 3, avenue Pierre Larousse, 92245 Malakoff Cedex, France.}\; and Marc Hoffmann\footnote{ENSAE-CREST and CNRS-UMR 8050, 3, avenue Pierre Larousse, 92245 Malakoff Cedex, France.}}

\date{}
\maketitle

\begin{abstract}
We investigate statistical inference across time scales. We take as toy model the estimation of  the intensity of a discretely observed compound Poisson process with symmetric Bernoulli jumps. We have data at times $i\Delta$ for $ i=0, 1,\ldots$ over $[0,T]$, for different sizes of $\Delta = \Delta_T$ relative to $T$ in the limit $T \rightarrow \infty$. We quantify the smooth statistical transition from a microscopic Poissonian regime ($\Delta_T \rightarrow 0)$ to a macroscopic Gaussian regime ($\Delta_T \rightarrow \infty$). The classical quadratic variation estimator is efficient in both microscopic and macroscopic scales but surprisingly shows a substantial loss of information in the intermediate scale $\Delta_T \rightarrow \Delta_\infty \in (0,\infty)$ that can be explicitly related to $\Delta_\infty$.
We discuss the implications of these findings beyond this idealised framework.

\end{abstract}

\section{Introduction} \label{introduction}
\subsection{Motivation}
We specialise in this paper on the example of a discretely observed compound Poisson process with symmetric Bernoulli jumps. This toy model is central to several application fields, {\it e.g.} financial econometrics or traffic networks (see the discussion in Section \ref{cas NH} and the references therein). Moreover, it already contains several interesting properties that enlight a tentative concept of statistical inference across scales. Consider a $1$-dimensional random process $(X_t)$ defined by
\begin{equation} \label{first Poisson compose}
X_t = X_0+\sum_{i = 1}^{N_t} \varepsilon_i,\;\;t\geq 0,
\end{equation}
where the $\varepsilon_i \in \{-1,1\}$ are independent, identically distributed with
$$\PP(\varepsilon_i=-1)=\PP(\varepsilon_i=1)=\frac{1}{2},$$
and independent of the standard homogeneous Poisson process $(N_t)$ with intensity $\vartheta \in \Theta = (0,\infty)$. Suppose we have discrete data over $[0,T]$ at times $i\Delta$.
This means that we observe
\begin{equation} \label{def observation}
\boldsymbol{X} = \big(X_0, X_{\Delta},\ldots, X_{\lfloor T\Delta^{-1}\rfloor}\big),
\end{equation}
and we obtain a statistical experiment by taking $\PP_\vartheta$ as the law of $\boldsymbol{X}$ defined by \eqref{def observation} when $(X_t)$ is governed by \eqref{first Poisson compose}.

On the one hand, if we observe $(X_t)$ {\it microscopically}, that is if $\Delta = \Delta_T \rightarrow 0$ as $T \rightarrow \infty$, then asymptotically, we can -- essentially -- locate the jumps of $(N_t)$ that convey all the relevant information about the parameter $\vartheta$.

In that case, $\boldsymbol{X}$ is ``close" to the continuous path $(X_t, t \in [0,T])$.
On the other hand, if we observe $(X_t)$ {\it macroscopically}, that is if $\Delta_T \rightarrow \infty$ under the constraint\footnote{This condition ensures that asymptotically infinitely many observations are recorded in the limit $T \rightarrow \infty$.} $T/\Delta_T\rightarrow \infty$ , we have a completely different picture: the diffusive approximation
\begin{equation} \label{diffusive approximation}
X_{i\Delta_T} - X_{(i-1)\Delta_T} \approx \sqrt{\vartheta \Delta_T} \big(W_{i\Delta_T} -W_{(i-1)\Delta_T}\big),
\end{equation}
becomes valid, where $(W_t)$ is a standard Wiener process. Inference on $\vartheta$ essentially transfers into a Gaussian variance estimation problem;
in that case, the state space rather becomes $\R^{\lfloor T\Delta^{-1}\rfloor+1}$.
Finally if we observe $(X_t)$ in the {\it intermediate scale} $0< \liminf\Delta_T \leq \limsup \Delta_T < \infty$, we observe a process presenting too many jumps to be located accurately from the data, and too few to verify the Gaussian approximation (\ref{diffusive approximation}). Therefore, depending on the scale parameter $\Delta_T$, the state space may vary, and it has an impact on the underlying random scenarios $\PP_\vartheta$, although the interpretation of the {\it parameter of interest} $\vartheta$ remains the same at all scales. What we have is rather a family of experiments
\begin{equation} \label{model}
{\mathcal E}^{T,\Delta} = \{\PP_{\vartheta}^{T,\Delta}, \vartheta \in \Theta\},
\end{equation}
where $\PP_\vartheta^{T,\Delta}$ denotes the law of $\boldsymbol{X}$ given by \eqref{def observation}
and these experiments ${\mathcal E}^{T,\Delta} $ may exhibit different behaviours at different scales $\Delta$. Heuristically, we would like to state that in the {\it microscopic scale} $\Delta_T \rightarrow 0$, the measure $\PP_\vartheta^{T,\Delta_T}$ conveys the same information about $\vartheta$ as the law of
\begin{equation} \label{microscopic approx exp}
(N_t, t \in [0,T]),
\end{equation}
that is if the jump times of $(X_t)$ were observed.
On the other side, in the {\it macroscopic scale} $\Delta_T \rightarrow \infty$ with $T/\Delta_T\rightarrow \infty$, the measure $\PP_\vartheta^{T,\Delta_T}$ shall convey the same information about $\vartheta$ as the law of
\begin{equation} \label{macroscopic approx exp}
\big(0,\sqrt{\vartheta}W_{\Delta_T},\ldots, \sqrt{\vartheta}W_{\lfloor T\Delta_T^{-1}\rfloor}\big),
\end{equation}
that is if the data were drawn as a Brownian diffusion with variance $\vartheta$. The following questions naturally arise:
\begin{itemize}
\item[{\bf i)}] How does the model formulated in \eqref{model} interpolate -- from a statistical inference perspective -- from microscopic (when $\Delta = \Delta_T \rightarrow 0$) to macroscopic scales (when $\Delta = \Delta_T \rightarrow \infty$)? In particular, how do intrinsic statistical information indices (such as the Fisher information) evolve as $\Delta = \Delta_T$ varies?
\item[{\bf ii)}] Is there any nontrivial phenomenon that occurs in the intermediate regime
$$0 < \liminf \Delta_T \leq \limsup \Delta_T <\infty?$$
\item[{\bf iii)}] Given {\bf i)} and {\bf ii)}, if a statistical procedure is optimal on a given scale $\Delta$, how does it perform on another scale? Is it possible to construct a single procedure that automatically adapts to each scale $\Delta$, in the sense that it is efficient simultaneously over different time scales?
\end{itemize}
\subsection{Main results}
In this paper, we systematically explore questions i), ii) and iii) in the simplified context of the experiments
${\mathcal E}^{T,\Delta}$ built upon the continuous time random walks model \eqref{first Poisson compose} for transparency. Some extensions to non-homogeneous compound Poisson processes are given, and the generalisation to a more general compound law is also discussed. As for i), we prove in Theorems \ref{structure statistique inter}, \ref{structure statistique micro} and \ref{structure statistique macro} that the LAN condition (Locally Asymptotic Normality\footnote{Recommended references are the textbooks \cite{IH} and \cite{VdV}, but we recall some definitions in Section \ref{resultats de regularite} for sake of completeness.}) holds for all scales $\Delta$. This means that $\PP_\vartheta^{T,\Delta}$ can be approximated -- in appropriate sense -- by the law of a Gaussian shift. We derive in particular the Fisher information of ${\mathcal E}^{T,\Delta}$ and observe that it smoothly depends on the scale $\Delta$. We shall see that the answer to ii) is positive. More precisely, we first prove in Theorem \ref{estimateurs} that the normalised quadratic variation estimator
$$\widehat \vartheta^{QV}_T = \frac{1}{T}\sum_{i = 1}^{\lfloor T \Delta_T^{-1}\rfloor}\big(X_{i\Delta_T}-X_{(i-1)\Delta_T}\big)^2$$
is asymptotically efficient -- it is asymptotically normal and its asymptotic variance is equivalent to the inverse of the Fisher information -- in both microscopic and macroscopic regimes. In the microscopic regime, it stems from the fact that the approximation
$$\widehat \vartheta^{QV}_T \approx \frac{1}{T}\sum_{0 \leq t \leq T} \big(X_t - X_{t^-}\big)^2 = \frac{N_T}{T}$$
becomes valid, as the jumps are $\pm 1$, and the efficiency is then a consequence of $N_T/T$ being the maximum likelihood estimator in the approximation experiment \eqref{microscopic approx exp}. In the macroscopic regime, thanks to the diffusive approximation \eqref{diffusive approximation} we have
$$\widehat \vartheta^{QV}_T \approx \frac{1}{T}\sum_{i = 1}^{\lfloor T \Delta_T^{-1}\rfloor} \big(\sqrt{\vartheta}(W_{i\Delta_T} - W_{(i-1)\Delta_T})\big)^2,$$
which is precisely the maximum likelihood estimator in the macroscopic approximation experiment \eqref{macroscopic approx exp}. Surprisingly, $\widehat \vartheta^{QV}_T$ fails to be efficient when
\begin{equation} \label{intermediate regime}
\Delta_T \rightarrow \Delta_\infty \in (0,\infty).
\end{equation}
More precisely, we show in Theorem \ref{cas intermediaire} that, although rate optimal, $\widehat \vartheta^{QV}_T$ misses the optimal variance by a non-negligible factor, depending on $\Delta_\infty$, that can reach up to $23\%$. This phenomenon is due to the fact that  in the intermediate regime \eqref{intermediate regime}, the process $(X_t)$ is sampled at a rate which has the same order as the intensity of its jumps.
 On the one hand, $(X_{i\Delta_T}-X_{(i-1)\Delta_T})^2$ gives no accurate information whereas a jump has occured or not during the period $[(i-1)\Delta_T, i\Delta_T]$, contrary to the case $\Delta_T \rightarrow 0$. On the other hand, there are not enough jumps to validate the approximation of $X_{i\Delta_T}-X_{(i-1)\Delta_T}$ by a Gaussian random variable, contrary to the case $\Delta_T \rightarrow \infty$. Finally, we construct in Theorem \ref{cas 1-Step} a one-step correction of $\widehat \vartheta^{QV}_T$ that provides an estimator efficient in all scales, giving a positive answer to iii).

The paper is organised as follows. We first propose in Section \ref{building up} a canonical framework for different time scales by considering the family of experiments $\big({\mathcal E}^{T, \Delta_T}\big)_{T>0}$. The way the scale parameter depends on $T$ defines the terms {\it microscopic, intermediate} and {\it macroscopic} scales rigorously.
Specialising to model \eqref{first Poisson compose} for transparency, the results about the structure of the corresponding $\big({\mathcal E}^{T, \Delta_T}\big)_{T>0}$ are stated in Section \ref{resultats de regularite}. We show in Theorems \ref{structure statistique inter}, \ref{structure statistique micro} and \ref{structure statistique macro} that the LAN (Local Asymptotic Normality) property holds simultaneously over all scales and provides an explicit expression for the Fisher information. The proof follows the classical route of \cite{IH} and boils down to obtaining accurate approximations of the distribution
$$f_{\Delta_T}(\vartheta, k) = \PP_\vartheta^{T,\Delta_T}\big(X_{i\Delta_T}-X_{(i-1)\Delta_T}=k\big),\;\;k\in\Z,$$
in the limit $\Delta_T\rightarrow 0$ or $\infty$. Note that $f_{\Delta_T}(\vartheta, k)$ does not depend on $i$ since $(X_t)$ has stationary increments. However explicit, the intricate form of $f_{\Delta_T}(\vartheta, k)$ requires asymptotic expansions of modified Bessel functions of the first kind. In the macroscopic regime however, we were not able to obtain such expansions. We take another route instead, proving directly the asymptotic equivalence in the Le Cam sense, a stronger result at the expense of requiring a rate of convergence of $\Delta_T$ to $\infty$, presumably superfluous. We show in Theorems \ref{estimateurs} and \ref{cas intermediaire} of Section \ref{distortion} that the quadratic variation estimator $\widehat \vartheta^{QV}_T$ is rate optimal and efficient in both microscopic and macroscopic regimes, but {\it not} in the intermediate scales \eqref{intermediate regime}. This negative result is however appended with the construction of an adaptive estimator, constructed by a standard one-step correction of $\widehat \vartheta_T^{QV}$ based on the likelihood at intermediate scales, and efficient over all scales (Theorem \ref{cas 1-Step}). Moreover this estimator has the advantage of being computationally implementable, contrary to the theoretical optimal maximum likelihood estimator. Section  \ref{cas NH} gives some extensions in the case of a non-homogeneous compound Poisson process (Theorem \ref{structure non-homogene}) and addresses the generalisation to more general compound laws. The comparison to related works on estimating L\'evy processes from discrete data is also discussed. Section \ref{preuves} is devoted to the proofs.

\section{Statement of the results} \label{main results}
\subsection{Building up statistical experiments across time scales} \label{building up}
Let $T>0$ and $\Delta >0$ be such that $\Delta \leq T$. On a rich enough probability space $(\Omega, {\mathcal F}, \PP)$, we observe the process $(X_t)$ defined in \eqref{first Poisson compose} at frequency $\Delta^{-1}$ over the period $[0,T]$.
Thus we observe $\boldsymbol{X}$ defined in \eqref{def observation},
and with no loss of generality\footnote{By assuming $X_0=0$, the first data point does not give information about the parameter $\vartheta$. If only asymptotic properties of the statistical model are studied, which is always the case here, it has no effect.}, we take $X_0=0$. We obtain a family of statistical experiments
$${\mathcal E}^{T,\Delta} := \big(\Z^{\lfloor T\Delta^{-1}\rfloor}, {\mathcal P}(\Z^{\lfloor T\Delta^{-1}\rfloor}), \big\{\PP_\vartheta^{T,\Delta},\vartheta \in \Theta\big\}\big),$$
where $\PP_\vartheta^{T,\Delta}$ denotes the law of $\boldsymbol{X}$ when $(X_t)$ has the form \eqref{first Poisson compose}, and $\Theta \subseteq (0,\infty)$ is a  parameter set with non empty interior. The experiment ${\mathcal E}^{T,\Delta}$ is dominated by the counting measure $\mu_T$ on $\Z^{\lfloor \Delta^{-1}T\rfloor}$.
Abusing notation slightly, we may\footnote{By taking for instance $\Omega=\Z^{\lfloor T\Delta_T^{-1}\rfloor}$.} (and will) identify
$\boldsymbol{X}$ with the canonical observation in ${\mathcal E}^{T,\Delta}$.
Since $(X_t)$ has stationary and independent increments under $\PP_\vartheta^{T,\Delta}$, we obtain the following expression for the likelihood
\begin{align*}
\frac{d\PP_\vartheta^{T,\Delta}}{d\mu_T}(\boldsymbol{X}) & = \prod_{i = 1}^{\lfloor T\Delta^{-1}\rfloor}f_{\Delta}(\vartheta, X_{i\Delta}-X_{(i-1)\Delta}),
\end{align*}
where we have set, for $k \in \Z$,
\begin{equation} \label{defdensite1}
f_{\Delta}(\vartheta,k):=\PP_\vartheta^{T,\Delta}\big(X_{i\Delta}-X_{(i-1)\Delta}=k\big) = \PP_\vartheta^{T,\Delta}\big(X_{\Delta}=k\big).
\end{equation}
We shall repeatedly use the terms {\it microscopic, intermediate} and {\it macroscopic} scale (or regime). In order to define these terms precisely, we let $\Delta = \Delta_T$ depend on $T$ with  $0 < \Delta_T \leq T$, and we adopt the following terminology.
\begin{definition}
The sub-family of experiments $({\mathcal E}^{T,\Delta_T})_{T>0}$ is said to be
\begin{enumerate}
\item On a microscopic scale (or regime) if $\Delta_T \rightarrow 0$ as $T \rightarrow \infty$.
\item On an intermediate scale (or regime) if $\Delta_T \rightarrow \Delta_\infty$ as $T \rightarrow \infty$, for some $\Delta_\infty \in (0,\infty)$.
\item On a macroscopic scale (or regime) if $\Delta_T \rightarrow  \infty$ and $T/\Delta_T \rightarrow \infty$ as $T \rightarrow \infty$.
\end{enumerate}
\end{definition}
\subsection{The regularity of $({\mathcal E}^{T,\Delta_T})_{T >0}$ across time scales} \label{resultats de regularite}
Let us recall\footnote{See for instance the textbooks \cite{IH} or \cite{VdV}.}
that the family of experiments $({\mathcal E}^{T,\Delta_T})_{T >0}$ satisfies the Local Asymptotic Normality (LAN) property at point $\vartheta \in \Theta$ with normalisation $I_{T,\Delta_T}(\vartheta)>0$ if, for every $v \in \R$ such that $\vartheta+vI_{T,\Delta_T}(\vartheta)^{1/2}\in \Theta$, the following decomposition holds
\begin{equation} \label{defLAN}
\frac{d\PP_{\vartheta+vI_{T,\Delta_T}(\vartheta)^{1/2}}^{T,\Delta_T}}{d\PP_{\vartheta}^{T,\Delta_T}}(\boldsymbol{X}) = \exp\big(v\, \xi_T-\frac{1}{2}v^2+r_T\big),
\end{equation}
where
\begin{equation} \label{defLAN bis}
\xi_T \rightarrow {\mathcal N}(0,1)\;\;\;\text{in distribution under}\;\;\PP_\vartheta^{T,\Delta_T}\;\;\text{as}\;\;T\rightarrow \infty
\end{equation}
and
\begin{equation} \label{defLAN ter}
r_T \rightarrow 0\;\;\;\text{in probability under}\;\;\PP_\vartheta^{T,\Delta_T}\;\;\text{as}\;\;T\rightarrow \infty.
\end{equation}
If \eqref{defLAN}, \eqref{defLAN bis} and \eqref{defLAN ter} hold, we informally say that $({\mathcal E}^{T,\Delta_T})_{T >0}$ is {\it regular with information} $I_{T,\Delta_T}(\vartheta)$. This means that locally around $\vartheta$, the law of $\boldsymbol{X}$ can be approximated by the law of a Gaussian shift experiment, where one observes a single random variable
$$\boldsymbol{Y} = \vartheta + I_{T,\Delta_T}(\vartheta)^{-1/2}\xi_T,$$
with $\xi_T$ being approximately distributed as a standard Gaussian random variable under $\PP_\vartheta^{T,\Delta_T}$ as $T\rightarrow \infty$. In particular, the optimal rate of convergence for recovering $\vartheta$ from $\boldsymbol{X}$ is the same as the one obtained from $\boldsymbol{Y}$. It is given by $I_{T, \Delta_T}(\vartheta)$ provided $I_{T,\Delta_T}(\vartheta) \rightarrow \infty$ as $T \rightarrow \infty$.
Note also that if the convergence of the remainder term $r_T = r_T(v)$ in \eqref{defLAN ter} holds locally uniformly in $v$, then  $I_{T,\Delta_T}(\vartheta)$ can be replaced by any function $J_{T,\Delta_T}(\vartheta)$ such that
$$J_{T,\Delta_T}(\vartheta) \sim I_{T,\Delta_T}(\vartheta)\;\;\;\text{as}\;\; T \rightarrow \infty$$
without affecting the LAN property. Hereafter, the symbol $\sim$ means asymptotic equivalence up to constants.
Our first result states the LAN property for the experiment $\big(\mathcal{E}^{T,\Delta}\big)_{T >0}$ on every scale $\Delta \in (0,\infty)$.
\begin{theorem}[The intermediate regime] \label{structure statistique inter}
Assume $\Delta_T \rightarrow \Delta_\infty \in (0,\infty)$ as $T \rightarrow \infty$. Then the family $({\mathcal E}^{T,\Delta_T})_{T >0}$ is regular and  we have
$$
I_{T,\Delta_T}(\vartheta) \sim I_{T, \Delta_\infty}(\vartheta) = T\Delta_\infty \Big(\E_{\vartheta}^{T,\Delta_\infty}\big[\big(h_{\Delta_\infty}(\vartheta, X_{\Delta_\infty})-(\vartheta\Delta_{\infty})^{-1}|X_{\Delta_\infty}|-1\big)^2\big]\Big),
$$
with
$$h_{\Delta_\infty}(\vartheta,k)=\frac{{\mathcal I}_{|k|+1}(\vartheta \Delta_\infty)}{{\mathcal I}_{|k|}(\vartheta \Delta_\infty)},$$
where, for $x \in \R$ and $\nu \in \N$,
$${\mathcal I}_{\nu}(x)=\sum_{m\geq 0}\frac{(x/2)^{2m+\nu}}{m! (\nu+m)!}$$
denotes the modified Bessel function of the first kind.
\end{theorem}
\begin{remark}
By taking $\Delta_T = \Delta_\infty \in (0,\infty)$ constant, we include the case of a fixed $\Delta$, therefore the same regularity result holds for $\big({\mathcal E}^{T,\Delta}\big)_{T > 0}$. An inspection of the proof of Theorem \ref{structure statistique inter} reveals that the mapping $\Delta \leadsto I_{T,\Delta}(\vartheta)$ is continuous over $(0,\infty)$.
\end{remark}
Our next result shows that formally we can let $\Delta_\infty \rightarrow  0$ in the expression of $I_{T,\Delta_\infty}(\vartheta)$ given by Theorem \ref{structure statistique inter} in the microscopic case. Moreover we obtain a simplified expression for the information rate.
\begin{theorem}[The microscopic case] \label{structure statistique micro}
Assume $\Delta_T \rightarrow 0$ as $T \rightarrow \infty.$
Then the family $({\mathcal E}^{T,\Delta_T})_{T >0}$ is regular and we have
$$I_{T,\Delta_T}(\vartheta) \sim I_{T,0}(\vartheta)  := \lim_{\Delta \rightarrow 0}I_{T, \Delta}(\vartheta) = \frac{T}{\vartheta}.$$
\end{theorem}
The macroscopic case is a bit more involved. In that case, we cannot formally let $\Delta_\infty \rightarrow \infty$ in the expression of $I_{T,\Delta_\infty}(\vartheta)$ given by Theorem \ref{structure statistique inter}. However we have the following simplification.
\begin{theorem}[The macroscopic case] \label{structure statistique macro}
Assume $\Delta_T \rightarrow \infty$,
$T/\Delta_T \rightarrow \infty$ as $T\rightarrow \infty$ and $T/\Delta_T^{1+\frac{1}{4}}=o\big((\log( T/\Delta_T))^{-\frac{1}{4}}\big)$. Then the family $({\mathcal E}^{T,\Delta_T})_{T >0}$ is regular and we have
$$I_{T,\Delta_T}(\vartheta) \sim I_{T,\infty}(\vartheta) := \frac{T\Delta_T^{-1}}{2\vartheta^2}.$$
\end{theorem}
The condition $T/\Delta_T^{1+\frac{1}{4}}=o\big((\log( T/\Delta_T))^{-\frac{1}{4}}\big)$ is technical but quite stringent; it stems from our method of proof, see Section \ref{preuve macro}. It is presumably superfluous, but we do not know how to relax it.
\subsection{The distortion of information across time scales} \label{distortion}

On each scale $\Delta >0$, let us introduce the empirical quadratic variation estimator
$$\widehat\vartheta_{T,\Delta}^{QV}=\frac{1}{T}\sum_{i=1}^{\left\lfloor T\Delta^{-1}\right\rfloor} (X_{i\Delta}-X_{(i-1)\Delta})^2$$
that mimics the behaviour of the maximum likelihood estimator in both macroscopic and microscopic regimes (see Section \ref{introduction}). More precisely, we have the following asymptotic normality result.
\begin{theorem} \label{estimateurs}
Let $\Delta = \Delta_T >0$ be such that $T/\Delta_T \rightarrow \infty$ as $T \rightarrow \infty$. We have
$$\widehat\vartheta_{T,\Delta_T}^{QV}=\vartheta+\big(I_{T,0}(\vartheta)^{-1}+ I_{T,\infty}(\vartheta)^{-1}\big)^{1/2}\xi_T,$$
where $\xi_T \rightarrow \mathcal{N}(0,1)$ in distribution under $\PP_{\vartheta}^{T,\Delta}$, and $I_{T,0}(\vartheta)$ and $I_{T,\infty}(\vartheta)$ are the information of the microscopic and macroscopic experiments given in Theorems \ref{structure statistique micro} and \ref{structure statistique macro} respectively.
\end{theorem}
On a microscopic scale $\Delta_T \rightarrow 0$, we have
$$\frac{I_{T,\infty}(\vartheta)^{-1}}{I_{T,0}(\vartheta)^{-1}} \rightarrow 0\;\;\text{as}\;\;T \rightarrow \infty.$$
On a macroscopic scale $\Delta_T\rightarrow \infty$ with $T/\Delta_T \rightarrow \infty$, we have on the contrary
$$\frac{I_{T,0}(\vartheta)^{-1}}{I_{T,\infty}(\vartheta)^{-1}} \rightarrow 0\;\;\;\text{as}\;\;T \rightarrow \infty.$$ As a consequence, we readily see
that  $\widehat\vartheta_{T,\Delta_T}^{QV}$ is asymptotically normal and that its asymptotic variance is equivalent to $I_{T,0}(\vartheta)^{-1}$ on a microscopic scale and to $I_{T,\infty}(\vartheta)^{-1}$ on  a macroscopic scale. At a heuristical level, this phenomenon can be explained directly by the form of the empirical quadratic variation estimator, as we already did in Section \ref{introduction}. At intermediate scales however, this is no longer true.
\begin{theorem}[Loss of efficiency in the intermediate regime] \label{cas intermediaire}
Assume that
\begin{equation} \label{restriction borne inf}
0 < \liminf \Delta_T \leq \limsup \Delta_T \leq \frac{1}{4\vartheta}.
\end{equation}
Then
$$\liminf_{T\rightarrow \infty}\frac{I_{T,0}(\vartheta)^{-1}+I_{T,\infty}(\vartheta)^{-1}}{I_{T,\Delta_T}(\vartheta)^{-1}} >1,$$
where $I_{T,\Delta_T}(\vartheta)$ is defined in Theorem \ref{structure statistique inter}.
\end{theorem}
\begin{remark}
For technical reasons, we are unable to prove that Theorem \ref{cas intermediaire} remains valid beyond the restriction $\limsup \Delta_T \leq 1/(4\vartheta)$. Numerical simulations suggest however that Theorem \ref{cas intermediaire} is valid whenever $\lim\sup\Delta_T<\infty$, see Figure \ref{Fig Ratio}.
\end{remark}
Let us denote by
$${\mathcal R}_T\big(\widehat \vartheta_{T, \Delta_T}^{QV},\vartheta\big) = \E_{\vartheta}^{T,\Delta_T}\big[\big(\widehat \vartheta_{T,\Delta_T}^{QV}-\vartheta\big)^2\big]$$ the squared error loss of the quadratic variation estimator. By Theorems \ref{structure statistique inter}, \ref{structure statistique micro} and \ref{structure statistique macro}, the family $\big({\mathcal E}^{T,\Delta_T}\big)_{T>0}$ is regular in all regimes and we may apply the classical minimax lower bound of Hajek, see for instance Theorem 12.1 in \cite{IH}: we have, for any
$\vartheta_0 \in \Theta$ and $\delta >0$ such that $[\vartheta_0-\delta, \vartheta_0+\delta] \subset \Theta$
\begin{equation} \label{borneinf}
\liminf_{T \rightarrow \infty}\sup_{|\vartheta - \vartheta_0| \leq \delta} I_{T,\Delta_T}(\vartheta){\mathcal R}_T\big(\widehat \vartheta_{T,\Delta_T}^{QV},\vartheta\big) \geq 1.
\end{equation}
On the one hand, Theorem \ref{estimateurs} suggests\footnote{This is actually true as the uniform integrability of $\widehat \vartheta_{T,\Delta_T}^{QV}$ under $\PP_\vartheta^{T,\Delta}$, locally uniformly in $\vartheta$, can easily be obtained. We leave the details to the reader.} that the lower bound \eqref{borneinf} can be achieved in microscopic and macroscopic regimes. On the other hand, Theorem \ref{cas intermediaire} shows that inequality \eqref{borneinf} is strict in the intermediate case, whenever the restriction \eqref{restriction borne inf} is satisfied, thus revealing a loss of efficiency in this sense.
Define
$$
\varphi(\vartheta, \Delta) = \E_\vartheta^{T,\Delta}\big[\big(h_{\Delta}(\vartheta, X_{\Delta})-(\vartheta\Delta)^{-1}|X_{\Delta}|-1\big)^2\big],$$
where $h_{\Delta}(\vartheta,k)$ is defined in Theorem \ref{structure statistique inter}. An inspection of the proof of Theorem \ref{cas intermediaire} shows that $\varphi(\vartheta, \Delta) = \psi(\vartheta \Delta)$, for some univariate function $\psi$, and that
$$\frac{I_{T,0}(\vartheta)^{-1}+I_{T,\infty}(\vartheta)^{-1}}{I_{T,\Delta}(\vartheta)^{-1}} = \psi(\vartheta \Delta)\big(2(\vartheta\Delta)^2+\vartheta \Delta\big).$$
The maximal loss of information is obtained for
$$\Delta^\star(\vartheta) \sim \vartheta^{-1}\mathrm{argmax}_{x >0} \psi(x)\big(2x^2+x\big)$$
as $T \rightarrow \infty$.
Numerical simulations show that the maximum loss of efficiency is close to $23\%$.
\begin{figure}
\begin{center}
\includegraphics[width=8.5cm,height=6cm]{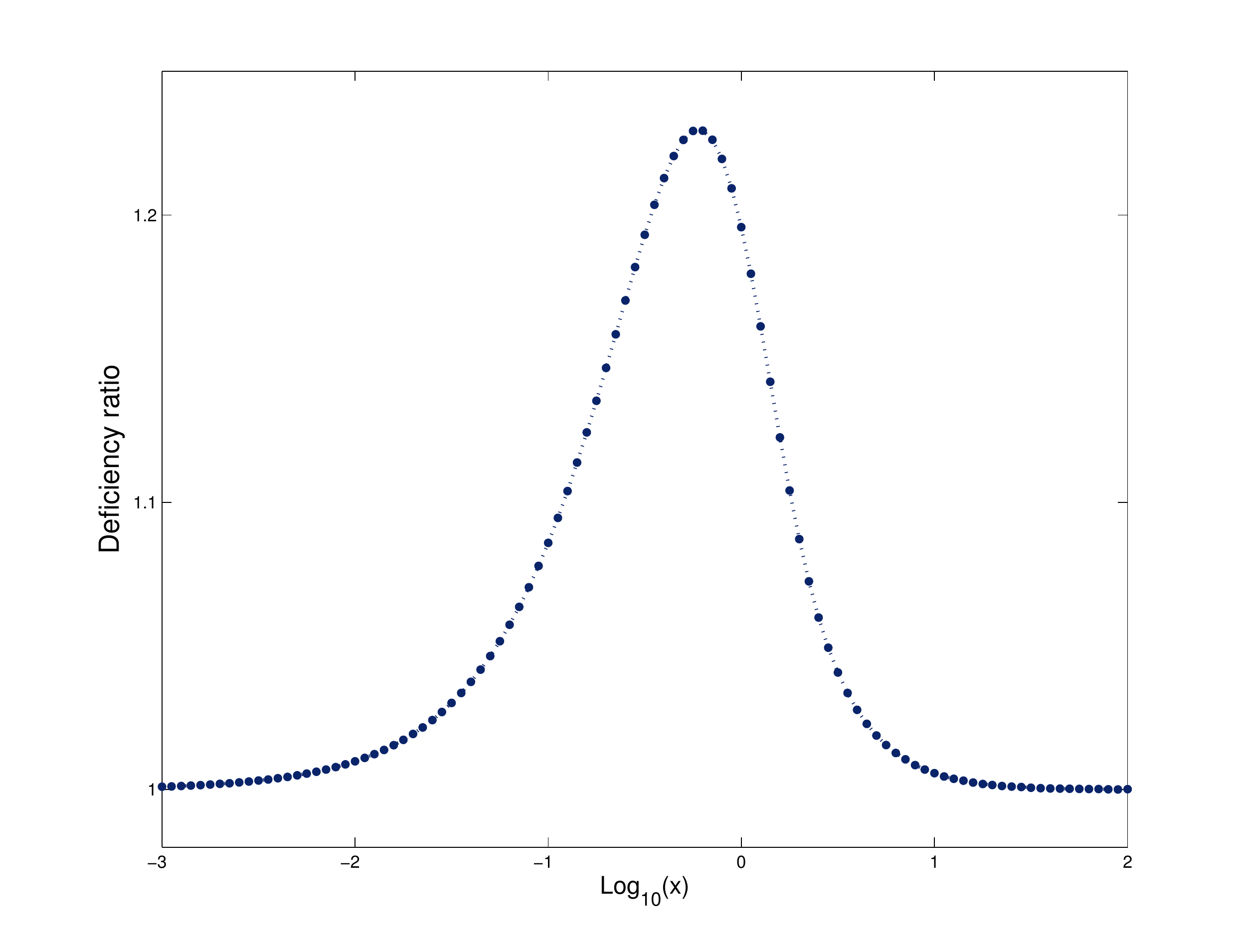}
\caption{\footnotesize{\textbf{Deficiency ratio through scales} ($x$-axis $x=\vartheta\Delta$, logarithmic scale). Ratio between the information $I_{T,\Delta_T}(\vartheta)$ and the inverse of the variance of $\widehat\vartheta^{QV}_{T,\Delta_T}$. The maximum is $1.2297$ and is attained at $x=0.600$.}}
\end{center}
\label{Fig Ratio}
\end{figure}
Since $\big({\mathcal E}^{T,\Delta}\big)_{T > 0}$ is regular for every $\Delta >0$, an asymptotically normal estimator with asymptotic variance equivalent to $I_{T,\Delta}(\vartheta)^{-1}$ is given by the maximum likelihood estimator. However due to the absence of a closed-form for the likelihood ratio that involves the intricate function $f_{\Delta}(\vartheta,k)$ defined in \eqref{defdensite1} (see also Section \ref{some estimates}), it seems easier to start from $\widehat\vartheta_{T,\Delta}^{QV}$ which is already rate-optimal by Theorem \ref{estimateurs} and correct it by a classical one-step iteration based on the Newton-Rhapson method, see for instance the textbook \cite{VdV} pp. 71--75. To that end, define
\begin{equation} \label{OneStep estimator}
\widehat\vartheta_{T,\Delta}^{OS}=\widehat\vartheta_T^{QV}-\frac{\sum_{i = 1}^{\lfloor T\Delta_T^{-1}\rfloor} \partial_{\vartheta }\log f_{\Delta}\big(\widehat\vartheta_{T,\Delta}^{QV},X_{i\Delta}-X_{(i-1)\Delta}\big)}{\sum_{i = 1}^{\lfloor T \Delta_T^{-1}\rfloor} \partial^2_{\vartheta}\log f_{\Delta}\big(\widehat\vartheta_{T,\Delta}^{QV}, X_{i\Delta}-X_{(i-1)\Delta}\big)}.
\end{equation}
\begin{theorem}\label{cas 1-Step}
In all three regimes (microscopic, intermediate and macroscopic), we have
$$I_{T,\Delta_T}^{1/2}\big(\widehat\vartheta_{T,\Delta_T}^{OS}-\vartheta\big) \longrightarrow {\mathcal N}(0,1)\;\;\;\text{in}\;\;\PP_\vartheta^{T,\Delta_T}\text{-distribution as}\;\;T\rightarrow \infty.$$
\end{theorem}
\begin{proof} In essence the regularity of $f_{\Delta}\big(\widehat\vartheta_{T,\Delta}^{QV},X_{i\Delta}-X_{(i-1)\Delta}\big)$ enables to apply Theorem 5.45 of Van der Vaart \cite{VdV}.
\end{proof}
Theorem \ref{cas 1-Step} expresses the fact that $\widehat\vartheta_{T,\Delta_T}^{OS}$ automatically adapts to $I_{T,\Delta_T}$ and is therefore optimal across scales.

\begin{figure}
\begin{center}
 \includegraphics[width=8.5cm,height=6cm]{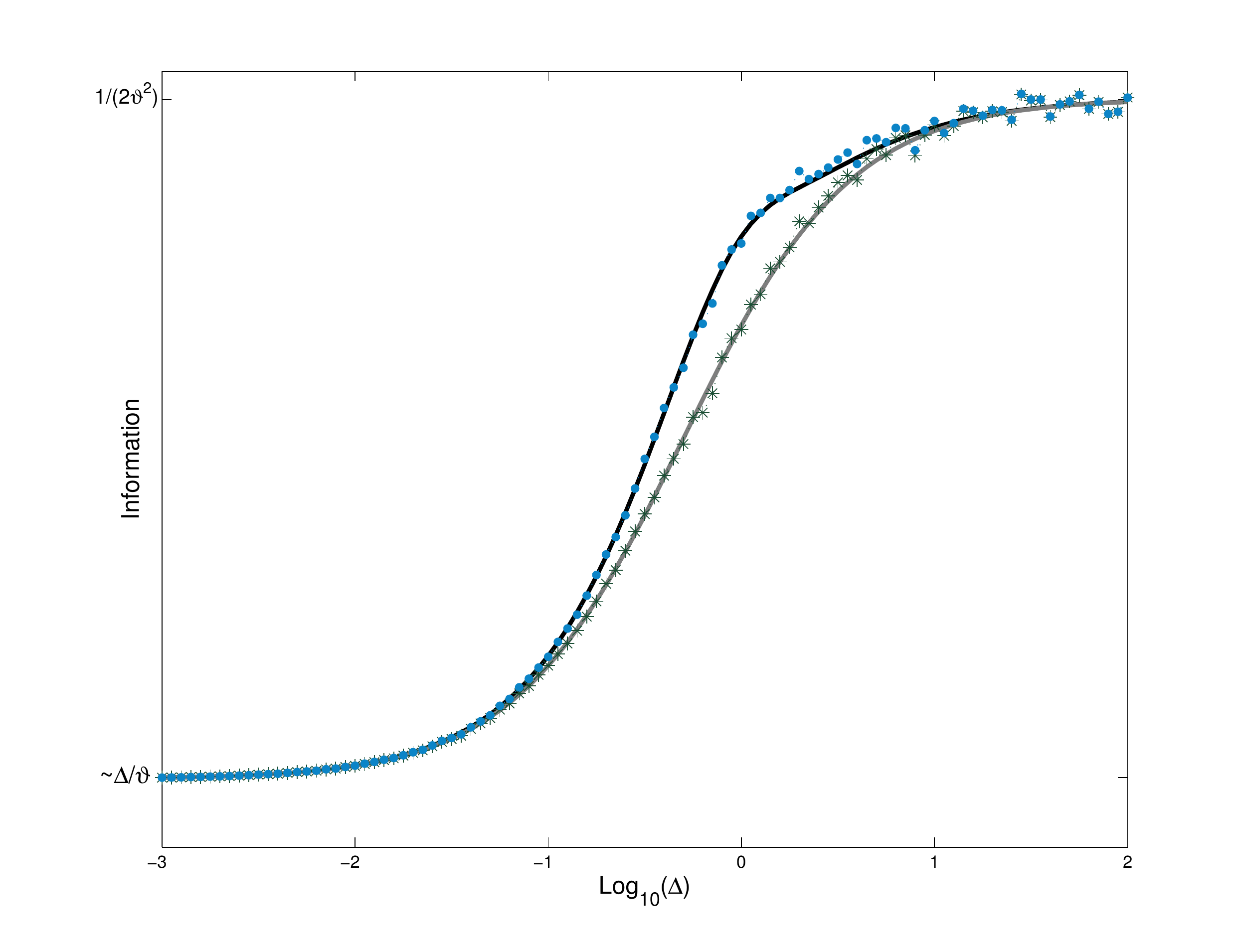}
 \caption{\label{pen} \footnotesize{\textbf{Information deficiency through scales ($\vartheta=1$, $x$-axis $\Delta$, logarithmic scale).} Information $I_{T,\Delta_T}$ (solid black). Inverse of the variance of $\widehat\vartheta^{QV}_{T,\Delta_T}$, theoretical (solid grey) and empirical (star green). The inverse of the empirical  variance of $\widehat\vartheta_{T,\Delta_T}^{OS}$ (dotted blue) is close to the optimal $I_{T,\Delta_T}$ on all scales. The empirical variances were computed using $10^4$ Monte-Carlo simulations from $T/\Delta = 10^5$ data, {\it i.e.} with the same number of data through scales.}}
 \end{center}
 \end{figure}
\section{Discussion} \label{cas NH}
The compound Poisson process $(X_t)$ with Bernoulli symmetric jumps defined in \eqref{first Poisson compose} is the simplest model of a continuous time symmetric random walk on a lattice that diffuses to a Brownian motion on a macroscopic scale. The intensity $\vartheta$ of the Poisson arrivals on a microscopic scale is transferred into the variance $\vartheta$ of the Brownian motion on a macroscopic scale:
\begin{equation} \label{homogeneous macro approx}
\Big(\frac{1}{\sqrt{T}}X_{tT},t \in [0,1]\Big) \longrightarrow \big(\sqrt{\vartheta} W_t, t \in [0,1]\big)
\end{equation}
in distribution as $T \rightarrow \infty$, where $(W_t)$ is a standard Brownian motion. The statistical inference program we have developed across time scales on the toy model given by $(X_t)$ can be useful in several applied fields. For instance, in financial econometrics, $(X_t)$ may be viewed as a toy model for a price process (last traded price, mid-price or bets bid/ask price) observed at the level of the order book, see {\it e.g.} \cite{Hautsch} or \cite{CTRW}. The parameter $\vartheta$ can be interpreted as a trading intensity on microscopic scales that transfers into a macroscopic volatility in the diffusion regimes. Our results convey the message that if a practitioner samples $(X_t)$ at high frequency at the same rate as price changes, which is customary in practice, then the realised volatility estimator $\widehat \vartheta_{T,\Delta_T}^{QV}$ is not efficient, and a modified estimator like $\widehat \vartheta_{T,\Delta_T}^{OS}$ should be used instead. However, this framework is a bit too simple and needs to be generalised in order to be more realistic in practice. Two directions can be explored in a relatively straightforward manner:
\begin{itemize}
\item[i)] The extension to a non-homogeneous intensity Poisson process.
\item[ii)] The extension to an arbitrary compound law on a discrete lattice.
\end{itemize}
\subsection*{Extension to the non-homogeneous case}
Theorems \ref{structure statistique inter}, \ref{structure statistique micro} and \ref{structure statistique macro} extend to the non-homogeneous case, when one allows the intensity of the jumps to depend on time. In this setting, the counting process $(N_t)$ defined in \eqref{first Poisson compose} is defined on $[0,T]$ and has intensity
$$\Lambda_T\big(t,\vartheta) = \int_0^{t} \lambda(\vartheta,\tfrac{s}{T})ds,\;\;\;\text{for}\;\;t \in [0,T]$$
where
$$\lambda: \vartheta \times [0,1] \rightarrow (0,\infty)$$
is the nonvanishing (integrable) intensity function, so that the process
$$\big(N_t-\Lambda_T(t,\vartheta), t \in [0,T]\big)$$ is a martingale. The homogenous case is recovered by setting $\lambda(\vartheta, t)=\vartheta$ for every $t\in [0,1]$. In this context, the macroscopic approximation \eqref{homogeneous macro approx} becomes
$$\Big(\frac{1}{\sqrt{T}}X_{tT},t \in [0,1]\Big) \longrightarrow \Big(\int_0^t\sqrt{\lambda(\vartheta,s)}\, dW_s, t \in [0,1]\Big)$$
in distribution as $T \rightarrow \infty$. We state -- without proof -- an extension of Theorems  \ref{structure statistique inter}, \ref{structure statistique micro} and \ref{structure statistique macro} for the associated family of experiments $\big({\mathcal E}^{T,\Delta_T}\big)_{T >0}$ across scales.

\begin{assumption} \label{hypothese}
 We have that $\vartheta \leadsto \lambda(\vartheta, t)$ is continuously differentiable for almost all $t\in [0,1]$ and moreover $\sup_{\vartheta \in \Theta, t \in [0,1]}\lambda(\vartheta, t) <\infty$.
\end{assumption}
\begin{theorem} \label{structure non-homogene}
We have Theorems \ref{structure statistique inter}, \ref{structure statistique micro} and \ref{structure statistique macro} with the following generalisation
\begin{enumerate}
\item In the microscopic case $\Delta_T \rightarrow 0$,
$$I_{T,0}(\vartheta) = T \int_0^1 \big(\partial_\vartheta \log \lambda(\vartheta, s)\big)^2 \lambda(\vartheta, s)ds.$$
\item In the intermediate regime $\Delta_T \rightarrow \Delta_\infty \in (0,\infty)$,
$$
I_T(\vartheta) = T\Delta_{\infty}\int_0^1\big(\partial_{\vartheta}\log \lambda(\vartheta,s)\big)^2 \lambda(\vartheta,s)^2 H(\vartheta, s)ds,$$
with
$$H(\vartheta, s) = \E_{\vartheta}^{T,\Delta_\infty}\big[\big(h_{\Delta_\infty}(\vartheta,s,X_{\Delta_{\infty}})+\big(\lambda(\vartheta,s)\Delta_{\infty}\big)^{-1}|X_{\Delta_{\infty}}| -1\big)^2\big],$$
and
$$h_{\Delta_\infty}(\vartheta,s,k)=\frac{{\mathcal I}_{|k|+1}(\lambda(\vartheta,s)\Delta_{\infty})}{{\mathcal I}_{|k|}(\lambda(\vartheta,s)\Delta_{\infty})}.$$
\item In the macroscopic case $\Delta_T \rightarrow \infty$ with $T/\Delta_T \rightarrow \infty$ and $T/\Delta_T^{1+\frac{1}{4}}=o((\log( T/\Delta_T))^{-\frac{1}{4}})$,
$$I_T(\vartheta) = \frac{T\Delta_T^{-1}}{2}\int_0^1\big(\partial_{\vartheta}\log\lambda(\vartheta,s)\big)^2ds.$$
\end{enumerate}
\end{theorem}
The proof of Theorem \ref{structure non-homogene} relies on the approximation
$$\int_{(i-1)\Delta_T}^{i\Delta_T}\lambda(\vartheta,s)ds=\Delta_T \lambda\Big(\vartheta,\frac{i-1}{T\Delta_T^{-1}}\Big)+\Delta_T r_T,$$
for $i=1,\ldots, \lfloor T\Delta_T^{-1}\rfloor$, where $r_T\rightarrow 0$ as $T \rightarrow \infty$ in all three regimes. Assumption \ref{hypothese} ensures that the convergence of the remainder is uniform in $i$ and $\vartheta$. This reduction enables us to transfer the problem of proving Theorems \ref{structure statistique inter}, \ref{structure statistique micro} and \ref{structure statistique macro} when substituting independent identically distributed random variables by independent non-equally distributed ones. This is not essentially more difficult, and the regularity of $\lambda$ enables us to piece together the local information given by each increment $X_{i\Delta_T}-X_{(i-1)\Delta_T}$ in order to obtain the formulae of Theorem \ref{structure non-homogene}.

An analogous program as in Section \ref{distortion} for the distortion of information could presumably be carried over, with appropriate modifications. For instance, one can show that
$$\widehat \vartheta_{T,\Delta_T}^{QV} = \sum_{i = 1}^{\lfloor T\Delta_T^{-1}\rfloor}\big(X_{i\Delta_T}-X_{(i-1)\Delta_T}\big)^2 \longrightarrow \int_0^1 \lambda(\vartheta,s)ds\;\;\;\text{as}\;\;T\rightarrow \infty$$
in $\PP_\vartheta^{T,\Delta_T}$-probability, in all three regimes. Then, in order to estimate $\vartheta$ efficiently, one should rather consider a contrast estimator that maximises
$$\widetilde \vartheta \leadsto U_{T,\Delta_T}(\widetilde \vartheta)=\sum_{i = 1}^{\lfloor T\Delta_T^{-1}\rfloor} g_{\Delta_T}\big(\lambda(\widetilde \vartheta, i\Delta_T), X_{i\Delta_T} - X_{(i-1)\Delta_T}\big)$$
for a suitable function $g_{T,\Delta_T}$, and make further assumptions on existence of a unique maximum for the limit -- whenever it exists -- of $U_{T,\Delta_T}$ under $\PP_\vartheta^{T,\Delta_T}$ as $T \rightarrow \infty$.
We do not pursue this here.
\subsection*{Extension to more general compound laws}
The situation is a bit more delicate when one tries to generalise Theorems \ref{structure statistique inter}, \ref{structure statistique micro} and \ref{structure statistique macro} to an arbitrary compound law $(\zeta(\vartheta, k), k \in \Z)$, for every $\vartheta \in \Theta$, with
$$0 \leq \zeta(\vartheta,k) \leq 1,\;\;\text{for}\;\;k \in \Z\;\;\text{and}\;\;\sum_{k \in \Z}\zeta(\vartheta, k)=1,$$
(and $\zeta(\vartheta, 0)=0$ for obvious identifiability conditions). We then observe a process $(X_t)$ of the form \eqref{first Poisson compose}, except that the jumps $(\varepsilon_i)$ are now distributed according to
$$\PP(\varepsilon_i=k)=\zeta(\vartheta,k),\;\;k\in\Z.$$
In order to keep up with the preceding case, we normalise the compound law, imposing
\begin{equation} \label{normalisation}
\sum_{k \in \Z}k\, \zeta(\vartheta, k)=0\;\;\;\text{and}\;\;\sum_{k \in \Z}k^2 \zeta(k,\vartheta)=1.
\end{equation}
First, in the microscopic case, we approximately observe over the period $[0,T]$ a random number of jumps, namely $N_T$ which is of order $\vartheta T$. Second, conditionally on $N_T$, the size of the jumps form a sequence of independent and identically distributed random variables with law $\zeta(\vartheta,k)$. On the other side, in the macroscopic limit, the effect of the size of the jumps is only tracked through their second moment, which is normalised to $1$ by \eqref{normalisation}. Therefore it gives no additional information about $\vartheta$. The situation is rather different from the case of symmetric Bernoulli jumps: here, the extraneous information about $\vartheta$ lies in the effect of the jumps, which are recovered in the microscopic regime and lost in the macroscopic one. There is however one way to reconcile with our initial setting, assuming that the compound law $\zeta(k)$ does not depend on $\vartheta$ and is known for simplicity.
Then, for $k \in \Z$, we have
$$f_{\Delta}(\vartheta, k) = \PP_\vartheta^{T,\Delta}\big(X_{i\Delta}-X_{(i-1)\Delta}=k\big) = \sum_{m \in \Z}\zeta^{\star m}(k)\frac{e^{-\vartheta \Delta}}{m!} (\vartheta \Delta)^m,$$
where $\zeta^{\star m}(k)$ is the probability that a random walk with law $\zeta(k)$ started at $0$ reaches $k$ in $m$ steps exactly. Therefore
\begin{equation}\label{gen law}
f_{\Delta}(\vartheta, k) = e^{-\vartheta \Delta} {\mathcal G}_k(\vartheta \Delta),\;\;\text{with}\;\;\;{\mathcal G}_k(x) =  \sum_{m \in \Z}\zeta^{\star m}(k) \frac{x^m}{m!}.
\end{equation}
In the symmetric Bernoulli case, we have ${\mathcal G}_k(x)={\mathcal I}_{|k|}(x)$, where ${\mathcal I}_{\nu}(x)$ is the modified Bessel function of the first kind. Anticipating the proof of Theorems \ref{structure statistique inter}, \ref{structure statistique micro} and \ref{structure statistique macro}, analogous results could presumably be obtained for an arbitrary compound law $\zeta(k)$ satisfying \eqref{normalisation}, provided accurate asymptotic expansions of ${\mathcal G}_k(x)$ are available in the viscinity of $0$ and $\infty$. The same subsequent results about the distortion of information that are developed in Section \ref{distortion} would presumably follow, with the same estimators $\widehat \vartheta_{T,\Delta_T}^{QV}$ and $\widehat \vartheta_{T,\Delta_T}^{OS}$, and the appropriate changes for $f_\Delta(\vartheta, k)$ in \eqref{OneStep estimator}.

\subsection*{Relation to other works}
Concerning the estimation of the law of the jumps, say $\zeta$, we have an inverse problem. One tries to recover $\zeta$ from the observations of a compound Poisson process, the link between $\zeta$ and the law of the process being given by \eqref{gen law}. In the setting of positive compound laws, Buchmann and Gr\"ubel \cite{Buchmann03, Buchmann04} succeed to invert that relation and give an estimator of $\zeta$ in the discrete and continuous case. That method which consists in inverting (\ref{gen law}) is called decompounding. It was generalised by Bogsted and Pitts \cite{Pitts10} to renewal reward processes when the law of the holding times is known, inrestriction to the case of having positive jumps only.

The compound Poisson process is a pure jump L\'evy process that can be studied accordingly. Using the L\'evy-Khintchine formula, it is possible to estimate nonparametrically its L\'evy measure which is given by the product $\vartheta\times\zeta$ in that case. This strategy is exploited by van Es {\it et al.} \cite{VAN-ES07} for a known intensity. This estimation procedure does not restrict to compound Poisson processes and it includes the case of pure jump L\'evy processes in general. Nonparametric estimation of the L\'evy measure from high frequency data (that corresponds to our microscopic case $\Delta_T\rightarrow 0)$ is thoroughly studied in  Comte and Genon-Catalot \cite{Comte09} as well as in the intermediate regime (with $\Delta_T=\Delta_\infty$ fixed) in \cite{Comte10}. In that latter case, we also have the results of Neumann and Rei{\ss} \cite{Reiss}.
\section{Proofs} \label{preuves}
\subsection{Preparation}
\subsubsection{Some estimates for $f_\Delta(\vartheta,k)$} \label{some estimates}
We have, for $k \in \Z$:
\begin{align*}
f_{\Delta}(\vartheta,k ) & =  \PP_\vartheta^{T,\Delta}\big(X_{\Delta}=k\big) = \sum_{m \geq 0}\phi_m(k)\frac{e^{-\vartheta \Delta}}{m!}(\vartheta \Delta)^m
\end{align*}
where $\phi_m(k)$ is the probability that a symmetric random walk in $\Z$ started from $0$ has value $k$ after $m$ steps exactly:
$$\phi_m(k)=
\left\{
\begin{array}{ll}
0 & \text{if}\; |k| > m\;\text{or}\;|k|-m\;\text{is odd} \\
2^{-m}\displaystyle \binom{m}{\tfrac{1}{2}(m+|k|)}&\text{otherwise}.
\end{array}
\right.
$$
Let us introduce the modified Bessel function of the first kind\footnote{The function $x \leadsto {\mathcal I}_{\nu}(x)$ can also be defined as the solution to the differential equation
$$x^2\frac{d^2y}{dx^2}+x\frac{dy}{dx}-(x^2+\nu^2)y=0.$$}
$${\mathcal I}_{\nu}(x)=\sum_{m\geq 0}\frac{(x/2)^{2m+\nu}}{m! \Gamma(\nu+m+1)},$$
for every $x, \nu \in \R$, and
where
$$\Gamma(x)=\int_{0}^{+\infty}t^{x-1}e^{-t}dt$$
denotes the Gamma function.
Straightforward computations
show that
\begin{equation} \label{def densite}
f_{\Delta}(\vartheta, k)=\exp(-\vartheta \Delta)\, {\mathcal I}_{|k|}\big(\vartheta \Delta\big).
\end{equation}
See for instance \cite{Sato}, p. 21 Example 4.7.
We gather some technically useful properties of the function ${\mathcal I}_{\nu}(x)$ that we will repeatedly use in the sequel.
\begin{lemma} \label{lemme Bessel}
\begin{enumerate}
\item For every $x \in \R\setminus \{0\}$ and $\nu \in \R$, we have
\begin{equation} \label{equadiff}
\partial_x \,{\mathcal I}_{\nu}(x) = {\mathcal I}_{\nu+1}(x)+\frac{\nu}{x}{\mathcal I}_{\nu}(x).
\end{equation}
\item For every $\mu > \nu>-\frac{1}{2}$ and $x>0$, we have
\begin{equation} \label{decroissance}
{\mathcal I}_{\mu}(x) < {\mathcal I}_{\nu}(x).
\end{equation}
\end{enumerate}
\end{lemma}
\begin{proof} Property 1 can be found in the textbook of Watson \cite{Watson} and readily follows from the fact that $x \leadsto {\mathcal I}_{\nu}(x)$ is analytical with an infinite radius of convergence. Property 2 is less obvious and follows from Nasell \cite{Nasell}.
\end{proof}
\subsubsection{The Fisher information of  ${\mathcal E}^{T,\Delta}$} \label{calcul info Fisher}
For $i = 1,\ldots, \lfloor T\Delta^{-1}\rfloor$, let ${\mathcal E}_{i}^{T,\Delta}$ denote the experiment generated by the observation of the incerement $X_{i\Delta}-X_{(i-1)\Delta} $. Since $(X_t)$ has independent stationary increments, we have, for $k \in \Z$
$$\PP_{\vartheta}^{T,\Delta}\big(X_{i\Delta}-X_{(i-1)\Delta}=k\big) = \PP_{\vartheta}^{T,\Delta}\big(X_{\Delta}=k\big) =f_\Delta(\vartheta,k).$$
Using that $X_0=0$, it follows that
\begin{equation} \label{factorization}
{\mathcal E}^{T,\Delta} = \bigotimes_{i = 1}^{\lfloor T\Delta^{-1}\rfloor}{\mathcal E}_i^{T,\Delta}
\end{equation}
as a product of independent observations given by the increments $X_{i\Delta}-X_{(i-1)\Delta}$,  each experiment  ${\mathcal E}_{i}^{T,\Delta}$ being dominated by the counting measure on $\Z$ with density $f_{\Delta}(\vartheta,k)$ given by \eqref{def densite} that does not depend on $i$.  Moreover, ${\mathcal E}_{i}^{T,\Delta}$ has (possibly infinite) Fisher information given by
$$\mathfrak{I}_{\Delta}(\vartheta) = \E_\vartheta^{T,\Delta_T}\big[\big(\partial_\vartheta \log f_{\Delta}(\vartheta, X_{\Delta})\big)^2\big] = \sum_{k \in \Z}\frac{\big(\partial_\vartheta f_{\Delta}(\vartheta,k)\big)^2}{f_{\Delta}(\vartheta,k)}\leq +\infty$$
 which does not depend on $i$. We study the regularity of ${\mathcal E}^{T,\Delta}$ in the classical sense of Ibragimov and Hasminskii (see \cite{IH} p. 65).
\begin{definition}
The experiment ${\mathcal E}_i^{T,\Delta}$ is regular (in the sense of Ibragimov and Hasminskii) if
\begin{enumerate}
\item[$\mathrm{i)}$] The mapping $\vartheta \leadsto f_{\Delta}(\vartheta,k)$ is continuous on $\Theta$ for every $k \in \Z$.
\item[$\mathrm{ii)}$] The Fisher information is finite: $ \mathfrak{I}_{\Delta}(\vartheta)<+\infty$ for every $\vartheta \in \Theta$.
\item[$\mathrm{iii)}$] The mapping $\vartheta \leadsto \partial_\vartheta \big(f_{\Delta}(\vartheta,\cdot)^{1/2}\big)$ is continuous in $\ell^2(\Z)$.
\end{enumerate}
\end{definition}
\begin{lemma} \label{regularity}
The experiments ${\mathcal E}_i^{T,\Delta}$ are regular.
\end{lemma}
\begin{proof}
For every $k\in \Z$,
$f_{\Delta}(\vartheta, k)  =\exp(-\vartheta \Delta) {\mathcal I}_{|k|}(\vartheta \Delta)$,
therefore i) is readily satisfied since $\vartheta \in \Theta \subset (0,\infty)$ and $\Delta >0$. We also have $f_{\Delta}(\vartheta, k)>0$ for every $k \in \Z$, then $ \mathfrak{I}_{\Delta}(\vartheta)$ is well defined, but possibly infinite. In order to prove ii), we  write
\begin{align}
\partial_\vartheta \log  f_{\Delta}(\vartheta, X_{\Delta})&  = \partial_\vartheta \log \big(e^{-\vartheta \Delta} {\mathcal I}_{|X_{\Delta}|}(\vartheta \Delta)\big) \nonumber -\Delta +\frac{\partial_\vartheta {\mathcal I}_{|X_{\Delta}|}(\vartheta \Delta)}{{\mathcal I}_{|X_{\Delta}|}(\vartheta \Delta)}\\
& = \Delta\big(h_{\Delta}(\vartheta,X_{\Delta})+(\vartheta\Delta)^{-1}|X_{\Delta}|-1\big), \label{score centre}
\end{align}
where we have set, for every $k \in \Z$,
$$
h_{\Delta}(\vartheta, k) = \frac{{\mathcal I}_{|k|+1}(\vartheta \Delta)}{{\mathcal I}_{|k|}(\vartheta \Delta)}
$$
and used Property \eqref{equadiff}. It follows that
\begin{equation} \label{Fisher individ}
\mathfrak{I}_{\Delta}(\vartheta)  = \Delta^2\, \E_\vartheta^{T,\Delta}\big[\big(h_{\Delta}(\vartheta, X_{\Delta})+(\vartheta\Delta)^{-1}|X_{\Delta}|-1\big)^2\big].
\end{equation}
Moreover the function $|X_{\Delta}| \leadsto {\mathcal I}_{|X_{\Delta}|}(\vartheta \Delta)$ is decreasing (see \eqref{decroissance}), thus
\begin{equation} \label{good bound}
0 \leq h_{\Delta}(\vartheta, X_{\Delta}) \leq 1
\end{equation}
and since $X_{\Delta}$ has all moments under $\PP_\vartheta^{T,\Delta}$, we obtain ii). We proceed similarly for iii). First, for any $\vartheta \in \Theta$ and $\varepsilon$  such that $\vartheta+\varepsilon \in \Theta$, we have
$$\partial_\vartheta \big(f_{\Delta}(\vartheta+\varepsilon,k)^{1/2}\big)-\partial_\vartheta  \big(f_{\Delta}(\vartheta,k)^{1/2}\big) = \varepsilon\, \partial_\vartheta^2 \big(f_{\Delta}(\vartheta_\varepsilon,k)^{1/2}\big)$$
for some $\vartheta_\varepsilon \in [\vartheta, \vartheta + \varepsilon]$. Second, we write
$$\partial_\vartheta \big(f_{\Delta}(\vartheta_\varepsilon,k)^{1/2}\big) = f_{\Delta}(\vartheta_\varepsilon,k)^{1/2} \tfrac{1}{2}\partial_\vartheta \log f_{\Delta}(\vartheta_\varepsilon,k),$$
and, differentiating a second time, we obtain that $\partial_\vartheta^2 \big(f_{\Delta}(\vartheta_\varepsilon,k)^{1/2}\big)$ equals
$$f_{\Delta}(\vartheta_\varepsilon,k)^{1/2}\Big(\big(\tfrac{1}{2}\partial_\vartheta \log f_{\Delta}(\vartheta_\varepsilon,k)\big)^2+\tfrac{1}{2}\partial_\vartheta^2 \log f_{\Delta}(\vartheta_\varepsilon,k)\Big).$$
Therefore, taking square and summing in $k$, we derive
\begin{align}
& \sum_{k \in \Z} \Big( \partial_\vartheta \big(f_{\Delta}(\vartheta+\varepsilon,k)^{1/2}\big)-\partial_\vartheta \big(f_{\Delta}(\vartheta,k)^{1/2}\big) \Big)^2 \nonumber \\
 = &\;\varepsilon^2\, \E_{\vartheta_\varepsilon}^{T,\Delta}\Big[\Big(\big(\tfrac{1}{2}\partial_\vartheta \log  f_{\Delta}(\vartheta_\varepsilon,X_{\Delta})\big)^2+\tfrac{1}{2}\partial_\vartheta^2 \log f_{\Delta}(\vartheta_\varepsilon,X_{\Delta})\Big)^2\Big].  \label{continuite}
\end{align}
From ii), we have
\begin{equation} \label{def log derivee}
\partial_\vartheta \log f_\Delta(\vartheta,X_{\Delta}) =
 \Delta\big(h_{\Delta}(\vartheta, X_{\Delta})+(\vartheta\Delta)^{-1}|X_{\Delta}|-1\big),
\end{equation}
and this last quantity has moments of all orders under $\PP_\vartheta^{T,\Delta}$, locally uniformly in $\vartheta$.
Likewise, using \eqref{def log derivee} and \eqref{equadiff}, it is easily seen that
\begin{align*}
\partial_\vartheta^2 \log f_{\Delta}(\vartheta,X_{\Delta})  =
   \;&\Delta^2\frac{{\mathcal I}_{|X_{\Delta}|+2}(\vartheta \Delta)}{{\mathcal I}_{|X_{\Delta}|}(\vartheta \Delta)}+\Delta\vartheta^{-1}h_{\Delta}(\vartheta,X_{\Delta}) \\
-&\;\Delta^2h_{\Delta}(\vartheta,X_{\Delta})^2 -\vartheta^{-2}|X_{\Delta}|.
\end{align*}
Thus $\partial_\vartheta^2 \log f_{\Delta}(\vartheta,X_{\Delta})$ has moments of all orders under $\PP_\vartheta^{T,\Delta_T}$ locally uniformly in $\vartheta$, thanks to \eqref{decroissance} and \eqref{good bound}. The same property carries over to the term within the expectation in \eqref{continuite} and we thus obtain iii) by letting $\varepsilon \rightarrow 0$.
\end{proof}
By the factorisation \eqref{factorization}, we infer that ${\mathcal E}^{T,\Delta}$ has Fisher information
$$\mathfrak{I}_{T,\Delta}(\vartheta) =\lfloor T\Delta_T^{-1}\rfloor \mathfrak{I}_{\Delta}(\vartheta) =\lfloor T\Delta^{-1}\rfloor  \sum_{k \in \Z}\frac{\big(\partial_\vartheta f_{\Delta}(\vartheta,k)\big)^2}{f_{\Delta}(\vartheta,k)}$$
which is finite thanks to ii) of Lemma \ref{regularity}.
\begin{lemma} \label{equivalents info}
For every $\vartheta \in \Theta$, we have
\begin{equation} \label{new Fisher}
\mathfrak{I}_{T,\Delta}(\vartheta)  =  \lfloor T\Delta^{-1}\rfloor
\Delta^2 \big(\E_{\vartheta}^{T,\Delta}\big[\big(h_{\Delta}(\vartheta, X_{\Delta})-(\vartheta\Delta)^{-1}|X_{\Delta}|-1\big)^2\big]\big),
\end{equation}
Moreover
in the microscopic and intermediate regimes, we have
\begin{equation} \label{equivalent info}
\frac{\mathfrak{I}_{T,\Delta_T}(\vartheta)}{I_{T,\Delta_T}(\vartheta)}\rightarrow 1\;\;\;\text{as}\;\;T \rightarrow \infty.
\end{equation}
\end{lemma}
\begin{proof}[of Lemma \ref{equivalent info}]
In the course of the proof of Lemma \ref{regularity}, we have seen by \eqref{Fisher individ} that
$$\mathfrak{I}_{\Delta}(\vartheta)  = \Delta^2\, \E_\vartheta^{T,\Delta}\big[\big(h_{\Delta}(\vartheta, X_{\Delta})+(\vartheta\Delta)^{-1}|X_{\Delta}|-1\big)^2\big].$$
It follows that
\begin{align*}
\mathfrak{I}_{T,\Delta}(\vartheta)  & =  \lfloor T\Delta^{-1}\rfloor \Delta^2\, \E_\vartheta^{T,\Delta}\big[\big(h_{\Delta}(\vartheta, X_{\Delta})+(\vartheta\Delta)^{-1}|X_{\Delta}|-1\big)^2\big] \\
 &=  \lfloor T\Delta^{-1}\rfloor
\Delta^2 \Big(\E_{\vartheta}^{T,\Delta}\big[\big(h_{\Delta}(\vartheta, X_{\Delta})+(\vartheta\Delta)^{-1}|X_{\Delta}|\big)^2\big]\\
&+1-2\E_{\vartheta}^{T,\Delta}\big[h_{\Delta}(\vartheta, X_{\Delta})+(\vartheta\Delta)^{-1}|X_{\Delta}|\big]\Big).
\end{align*}
Since ${\mathcal E}^{T,\Delta}_i$ is regular by Lemma \ref{regularity}, we have $\E_\vartheta^{T,\Delta}\big[\partial_\vartheta \log f_{\Delta}(\vartheta, X_{\Delta})\big]=0$. Combining this with the equality
$$\partial_\vartheta \log  f_{\Delta}(\vartheta, X_{\Delta}) = \Delta\big(h_{\Delta}(\vartheta,X_{\Delta})+(\vartheta\Delta)^{-1}|X_{\Delta}|-1\big)
$$
that we obtained in \eqref{score centre},
we derive
$$\E_{\vartheta}^{T,\Delta}\big[h_{\Delta}(\vartheta, X_{\Delta})+(\vartheta\Delta)^{-1}|X_{\Delta}|\big]=1,$$
and \eqref{new Fisher} follows. Expanding \eqref{new Fisher} further, we obtain the useful representation
\begin{align}
 \mathfrak{I}_{T,\Delta}(\vartheta)
& =  \lfloor T\Delta^{-1}\rfloor \Big(
\Delta^2\, \E_{\vartheta}^{T,\Delta}\big[h_{\Delta}(\vartheta, X_{\Delta})^2\big] \nonumber \\
& +\frac{2\Delta}{\vartheta}\E_\vartheta^{T,\Delta}\big[|X_{\Delta}|h_{\Delta}(\vartheta, X_{\Delta})\big]+\frac{\Delta}{\vartheta}-\Delta^2
\Big).  \label{technical Fisher}
\end{align}
Let us now assume that $\Delta = \Delta_T \rightarrow 0$. We will need the following asymptotic expansion of the function ${\mathcal I}_\nu(x)$ near $0$.
\begin{lemma} \label{exp Bessel 0}
We have, for $\nu \in \N$,
\begin{equation} \label{dev Bessel en 0}
{\mathcal I}_{\nu}(x) = \frac{1}{2^{\nu}\nu!} x^{\nu}\Big(1+x
r_\nu(x)\Big), \\
\end{equation}
where $x \leadsto r_\nu(x)$ is continuous and satisfies $\sup_{\nu \geq 0}r_{\nu}(x)\rightarrow 0$ when $x\rightarrow 0$.
\end{lemma}
\begin{proof}[Proof of Lemma \ref{exp Bessel 0}]
We have an expression of $\mathcal{I}_{\nu}(x)$ as a power series, thus its Taylor expansion in a neighborhood of $0$ is given by
$$
\mathcal{I}_{\nu}(x)=\Big(\frac{x}{2}\Big)^{\nu}\frac{1}{\nu!}\Big(1+xr_{\nu}(x)\Big),
$$
where
$$r_{\nu}(x)=\frac{x}{4}\sum_{m\geq0}\frac{\nu!}{(m+1)!(m+1+\nu)!} \Big(\frac{x}{2}\Big)^{2m}\leq x\sum_{m\geq0}\frac{1}{m!}\Big(\frac{x}{2}\Big)^{2m} =xe^{x^2/2}.$$
Then $x \leadsto r_\nu(x)$ is continuous and satisfies $\sup_{\nu \geq 0}r_{\nu}(x)\rightarrow 0$ when $x\rightarrow 0$.
\end{proof}
By Lemma \ref{exp Bessel 0},  a simple Taylor expansion shows that
$$h_{\Delta_T}(\vartheta,X_{\Delta})= \frac{{\mathcal I}_{|X_{\Delta}|+1}(\vartheta \Delta_T)}{{\mathcal I}_{|X_{\Delta_T}|}(\vartheta \Delta_T)} = \frac{\Delta_T}{2\vartheta}\frac{1}{|X_{\Delta}|+1}+\Delta_T r'_T(\vartheta,X_{\Delta})$$
where $|r'_T(\vartheta, X_\Delta)| \leq c(\vartheta)$, for some deterministic locally bounded $c(\vartheta)$.
Plugging this last expression in \eqref{technical Fisher}, we obtain
$$\mathfrak{I}_T(\vartheta) = \frac{T}{\vartheta}+T\Delta_Tr''_T(\vartheta),$$
with $r''_T$ having the same property as $r_T$, whence \eqref{equivalent info} in the microscopic case. In the intermediate case, since $\Delta \leadsto \E_\vartheta^{T,\Delta}\big[\big(h_{\Delta}(\vartheta, X_{\Delta})+(\vartheta\Delta)^{-1}|X_{\Delta}|-1\big)^2\big]$ is continuous on $(0,\infty)$, we readily obtain the result using that $\lfloor T\Delta_T^{-1}\rfloor
\Delta_T^2$ is equivalent to $T\Delta_T$ as $T \rightarrow \infty$. The proof of Lemma \ref{equivalents info} is complete.
\end{proof}
\subsection{Proof of Theorems \ref{structure statistique inter} and \ref{structure statistique micro} }
A technically convenient consequence of Lemma \ref{equivalents info} in the microscopic and macroscopic cases is that it suffices to prove Theorems \ref{structure statistique inter} and \ref{structure statistique micro} with $\mathfrak{I}_{T,\Delta_T}(\vartheta)$ instead of $I_{T,\Delta_T}(\vartheta)$, provided the convergence \eqref{defLAN ter} is valid locally uniformly. As ${\mathcal E}^{T,\Delta_T}$ is the product of ${\mathcal E}_{i}^{T,\Delta_T}$ generated by the $X_{i\Delta_T}-X_{(i-1)\Delta_T}$ that form independent and identically distributed random variables under
$\PP_\vartheta^{T,\Delta_T}$ with distribution depending on $T$, we are in the framework of Theorem 3.1' p. 128 in Ibragimov and Hasminskii \cite{IH} and the LAN property is a consequence of the following two conditions:
\begin{itemize}
\item[$\mathrm{i)}$] For every $\vartheta_0 \in \Theta$ and $h$ such that $\vartheta_0 + h \in \Theta$, we have
$$
\frac{\lfloor T\Delta_T^{-1}\rfloor }{\mathfrak{I}_{T,\Delta_T}(\vartheta_0)^2}\sup_{|\vartheta -\vartheta_0| \leq \tfrac{h}{\mathfrak{I}_{T,\Delta_T}(\vartheta_0)^{1/2}}}\sum_{k \in \Z}\big|\partial^2_\vartheta \big(f(\vartheta,k)^{1/2}\big)\big|^2 \rightarrow 0
$$
as $T \rightarrow \infty$.
\item[$\mathrm{ii)}$] For every $h >0$ and $\vartheta \in \Theta$, we have
$$
\frac{\lfloor T\Delta_T^{-1}\rfloor}{\mathfrak{I}_{T,\Delta_T}(\vartheta)}  \E_\vartheta^{T,\Delta_T}\big[\big(\partial_\vartheta \log f(\vartheta, X_{\Delta_T})\big)^{2}{\bf 1}_{\{|\partial_\vartheta \log f(\vartheta, X_{\Delta})|\geq h \mathfrak{I}_{T,\Delta_T}(\vartheta)\}}\big]\rightarrow 0
$$
as $T \rightarrow \infty$.
\end{itemize}
In the same way as for the proof of ii) in Lemma \ref{regularity}, we have
\begin{align*}
\partial_\vartheta^2 f_{\Delta_T}(\vartheta,k)  &= \tfrac{1}{2} f_{\Delta_T}(\vartheta, k)^{1/2} \Big(\tfrac{1}{2}\big(-\Delta_T+|k|\vartheta^{-1}+\Delta_Th_{\Delta_T}(\vartheta,k)\big)^2 \\
&-|k|\vartheta^{-2}+\Delta_T^2\big(\frac{{\mathcal I}_{|k|+2}(\vartheta \Delta_T)}{{\mathcal I}_{|k|}(\vartheta \Delta_T)}-h_{\Delta_T}
(\vartheta, k)^2\big)+\frac{\Delta}{\vartheta}h_{\Delta_T}(\vartheta,k)\Big) \\
&= \tfrac{1}{2}f_{\Delta_T}(\vartheta, k)^{1/2}\mathcal{H}_{\Delta_T}(\vartheta, k),\;\;\;\text{say.}
\end{align*}
Therefore, taking squares and summing in $k$, i) is proved if we show that
\begin{equation} \label{reduction cond 1 LAN}
\frac{\lfloor T\Delta_T^{-1}\rfloor}{\mathfrak{I}_{T,\Delta_T}(\vartheta_0)^2} \sup_{|\vartheta -\vartheta_0| \leq \tfrac{h}{\mathfrak{I}_{T,\Delta_T}(\vartheta_0)^{1/2}}} \E_{\vartheta}^{T,\Delta_T}\big[\mathcal{H}_{\Delta_T}(\vartheta, X_{\Delta_T})^2\big]\rightarrow 0
\end{equation}
as $T \rightarrow \infty$. Using \eqref{decroissance}, we have
$$0 \leq h_{\Delta_T}(\vartheta, k)\leq 1\;\;\;\text{and}\;\;
0 \leq \frac{{\mathcal I}_{|k|+2}(\vartheta \Delta_T)}{{\mathcal I}_{|k|}(\vartheta \Delta_T)} \leq 1,$$
hence $\mathcal{H}_{\Delta}(\vartheta, X_{\Delta_T})^2$ is less than
$$c(\vartheta)\big(\Delta_T^4+\Delta_T(1+\Delta_T^2)X_{\Delta_T}+
(1+\Delta_T^2)X_{\Delta_T}^2+(1+\Delta_T)\big|X_{\Delta_T}\big|^3+X_{\Delta_T}^4\big)$$
for a locally bounded $c(\vartheta)$, which in turn is less than
\begin{equation} \label{borne a integrer}
c'(\vartheta,\Delta_T)\big(\Delta_T+\Delta_T^4+X_{\Delta_T}^2+X_{\Delta_T}^4\big),
\end{equation}
for some $c'(\vartheta, \Delta_T)$ locally bounded on $\Theta \times [0,\infty)$.
Since $(X_t)$ is a compound Poisson process under $\PP_{\vartheta}^{T,\Delta_T}$ with intensity $\vartheta$ and jumps in $\{-1,+1\}$ with equal probability, the characteristic function of $X_{\Delta_T}$ is explicitly given by
$$\E_{\vartheta}^{T,\Delta_T}\big[e^{iu X_{\Delta_T}}\big] = \exp\big(-\vartheta \Delta_T(1-\cosh u)\big),\;\;u \in \R,$$
from which we obtain
\begin{equation} \label{moment 4}
\E_{\vartheta}^{T,\Delta_T}[X_{\Delta}^4]=\vartheta \Delta_T(1+3\vartheta \Delta_T).
\end{equation}
Integrating \eqref{borne a integrer}, we derive
$$\E_{\vartheta}^{T,\Delta_T}\big[\mathcal{H}_{\Delta_T}(\vartheta, X_{\Delta_T})^2\big] \leq c''(\vartheta,\Delta_T)\Delta_T,$$
where $c''$ has the same property as $c'$. Since $\mathfrak{I}_{T,\Delta_T}(\vartheta_0)$ is of order $T$ as $T \rightarrow \infty$ in both microscopic and intermediate scales, we obtain \eqref{reduction cond 1 LAN} and i) follows.

It remains to prove ii). From the explicit representation
\begin{equation} \label{explicit representation}
\partial_\vartheta \log f_{\Delta_T}(\vartheta,k) = \Delta_T\big(-1+h_{\Delta_T}(\vartheta, k)\big)+|k|\vartheta^{-1}
\end{equation}
and the fact that $0 \leq h_{\Delta_T}(\vartheta,k) \leq 1$, we have
$$\big| \partial_\vartheta\log f_{\Delta_T}(\vartheta,k) \big| \leq \Delta_T + |k|\vartheta^{-1},$$
from which we readily obtain
$$\E_{\vartheta}^{T,\Delta_T}\big[\big( \partial_\vartheta\log f_{\Delta_T}(\vartheta,X_{\Delta_T}) \big)^2\big] \leq c'''(\vartheta, \Delta_T),$$
where $c'''$ has the same property as $c'$. Since $\mathfrak{I}_{T,\Delta_T}(\vartheta) \rightarrow \infty$ as $T\rightarrow \infty$, we obtain ii) in both microscopic and intermediate scales. The proof of Theorems \ref{structure statistique inter} and \ref{structure statistique micro} is complete.
\subsection{Proof of Theorem \ref{structure statistique macro}} \label{preuve macro}
The strategy of the proof is quite different from that of Theorems \ref{structure statistique inter} and \ref{structure statistique micro}, for we were not able to obtain asymptotic expansions for $\mathcal{I}_{\nu}(x)$ in a viscinity of $x=\infty$ with appropriate bounds on the stochastic remainders.

Consider instead the experiment ${\mathcal Q}^{T,\Delta_T} = \{\QQ_\vartheta^{T,\Delta_T}, \vartheta \in \Theta\}$ generated by the observation of $\lfloor T\Delta_T^{-1}\rfloor$ independent centred Gaussian random variables with variance $\vartheta\Delta_T$, for $\Delta_T \rightarrow \infty$ satisfying the rate restriction
\begin{equation} \label{rate restriction}
T/\Delta_T^{1+\frac{1}{4}}=o((\log( T/\Delta_T))^{-\frac{1}{4}})
\end{equation}
We plan to prove that under the restriction \eqref{rate restriction}, the experiments ${\mathcal E}^{T,\Delta_T}$ and ${\mathcal Q}^{T,\Delta_T}$ are asymptotically equivalent as $T \rightarrow \infty$. Theorem \ref{structure statistique macro} then follows from the Le Cam theory, see for instance \cite{LCY}. To that end, we map each increment $X_{i\Delta_T}-X_{(i-1)\Delta_T}$ in ${\mathcal E}^{T,\Delta_T}$ with
$$Y_{i}^{\Delta_T} = X_{i\Delta_T}-X_{(i-1)\Delta_T}+U_i$$
where the $U_i$ are independent random variables uniformly distributed on $[-\tfrac{1}{2},\tfrac{1}{2}]$. Let us denote by $\widetilde {\mathcal E}^{T,\Delta_T} = \{\widetilde \PP_\vartheta^{T,\Delta_T}, \vartheta \in \Theta\}$ the experiment generated by the $Y_{i}^{\Delta_T}$. Since the increments $X_{i \Delta}-X_{(i-1)\Delta_T}$ take values in $\Z$, we have a one-to-one correspondence between $Y_{i}^{\Delta_T}$ and the  increment $X_{i\Delta_T}-X_{(i-1)\Delta_T}$ and therefore the experiments ${\mathcal E}^{T,\Delta_T}$ and $\widetilde {\mathcal E}^{T,\Delta_T}$ are equivalent. Moreover, $\widetilde {\mathcal E}^{T,\Delta_T}$ and ${\mathcal Q}^{T,\Delta_T}$ live on the same state space $\R^{\lfloor T\Delta_T^{-1}\rfloor}$ and have smooth densities with respect to the Lebesgue measure. The proof of Theorem \ref{structure statistique macro} is therefore implied by the following bound
$$\|\widetilde \PP_\vartheta^{T,\Delta_T}- \QQ^{T,\Delta_T}_\vartheta\|_{TV} \rightarrow 0\;\;\text{as}\;\; T \rightarrow \infty,$$
locally uniformly in $\vartheta$ and where $\|\cdot\|_{TV}$ denotes the variational norm. This bound is implied in turn by the bound
\begin{equation} \label{dist var}
\|{\mathcal L}(Y_1^{\Delta_T}) - {\mathcal N}(0,\vartheta\Delta_T)\|_{TV} = o\big((T/\Delta_T)^{-1}\big)
\end{equation}
locally uniformly in $\vartheta$, since each experiment is the $\lfloor (T/\Delta_T)^{-1} \rfloor$-fold product independent and identically distributed random variables\footnote{For instance, by using the bound
$$\|\PP^{\otimes n}-\QQ^{\otimes n}\|_{TV} \leq \sqrt{2}\big(1-\big(1-\tfrac{1}{2}\|\PP-\QQ\|_{TV}\big)^n\big)^{1/2}.$$}.
Let us further denote by $p_{\vartheta,\Delta_T}$ and $q_{\vartheta, \Delta_T}$ the densities of $Y_1^{\Delta_T}$ and the Gaussian law ${\mathcal N}(0,\vartheta \Delta_T)$ respectively. We have
\begin{align*}
\|{\mathcal L}(Y_1^{\Delta_T}) - {\mathcal N}(0,\vartheta\Delta_T)\|_{TV}  & = \int_{\R}|p_{\vartheta,\Delta_T}(x)-q_{\vartheta,\Delta_T}(x)|dx \\
 & \leq I + II + III,
\end{align*}
where, applying sucessively the triangle inequality and Cauchy-Schwarz, for any $\eta >0$,
\begin{align*}
I & = \sqrt{2\eta}\big(\int_{\R}\big(p_{\vartheta,\Delta_T}(x)-q_{\vartheta,\Delta_T}(x)\big)^2dx\big)^{1/2}, \\
II & =  \PP_{\vartheta}^{\Delta_T}\big(|X_{\Delta_T}+U_1|\geq \eta\big), \\
III & = \int_{|x|\geq \eta}q_{\vartheta, \Delta_T}(x)dx.
\end{align*}
Set $\eta = \eta_T = \kappa\sqrt{\Delta_T \log(T/\Delta_T)}$. We claim that for $\kappa^2 > 2\vartheta$, the terms $I$, $II$ and $III$ are $o\big((T/\Delta_T^{-1})\big)$ hence \eqref{dist var} and the result, for an appropriate choice of $\kappa$ so that the convergence can hold locally uniformly in $\vartheta$. Since $q_{\vartheta, \Delta_T}(x)$ is the density of the Gaussian law ${\mathcal N}(0,\vartheta \Delta_T)$, we readily obtain
$$III \leq 2 \exp\big(-\frac{\eta_T^2}{2\vartheta \Delta_T}\big) = (T/\Delta_{T})^{-\kappa^2/(2\vartheta)} = o\big((T/\Delta_T^{-1})\big)$$
using $\kappa^2 >2\vartheta$. For the term $II$, we observe that since $|U_1| \leq 1/2$, we have
$$\PP_{\vartheta}^{\Delta_T}\big(|X_{\Delta_T}+U_1|\geq \eta_T\big) \leq \PP_{\vartheta}^{\Delta_T}\big(\big|\sum_{i = 1}^{N_{\Delta_T}}\varepsilon_i\big| \geq \eta_T-\tfrac{1}{2}\big),$$
where the $\varepsilon_i \in\{-1,1\}$ are independent and symmetric. By Hoeffding inequality, this term is further bounded by
$$2\E_\vartheta^{T,\Delta_T}\big[\exp\big(-\tfrac{(\eta_T-1/2)^2}{2N_{\Delta_T}}\big)\big] \leq 2\Big(\exp\big(-\tfrac{(\eta_T-1/2)^2}{2\kappa' \Delta_T}\big)+\PP_\vartheta^{T,\Delta_T}\big(N_{\Delta_T} \geq \kappa' \Delta_T\big)\Big)$$
for every $\kappa'>0$. If $\kappa' <\kappa^2/2$, one readily checks that
$$\exp\big(-\tfrac{(\eta_T-1/2)^2}{2\kappa' \Delta_T}\big) = o\big((T/\Delta_T)^{-1}\big).$$
Moreover, if $\kappa' > \vartheta$, we have, by Chernov inequality,
$$\PP_\vartheta^{T,\Delta_T}\big(N_{\Delta_T}-\vartheta \Delta_T \geq (\kappa'-\vartheta)\Delta_T\big) \leq \exp\Big(-\Delta_T(\kappa'\log(\kappa'/\vartheta)-(\kappa'-\vartheta)\big)\Big)$$
and this term is also $o\big((T/\Delta_T)^{-1}\big)$. Thus $II$ and $III$ have the right order and it remains to bound the main term $I$. By Plancherel equality we obtain the following explicit expression:
\begin{align*}
& \int_{\R} \big(p_{\vartheta,\Delta_T}(x)- q_{\vartheta,\Delta_T}(x)\big)^2dx =  (2\pi)^{-1}\int_{\R} (\widehat{p}_{\vartheta,\Delta_T}(\xi)- \widehat{q}_{\vartheta,\Delta_T}(\xi))^2d\xi\\
 = &\; (2\pi)^{-1}\int_{\R}\Big(e^{-\vartheta\Delta_T(1-\cos\xi)}\frac{\sin\tfrac{\xi}{2}}{\tfrac{\xi}{2}}- e^{-\tfrac{1}{2}\vartheta \Delta_T\xi^2}\Big)^2d\xi \\
 = &\; (2\pi)^{-1}\int_{\R}\Big(e^{-\vartheta\Delta_T\big(1-\cos(\tfrac{\xi}{\sqrt{\Delta_T}})\big)}\frac{\sin(\tfrac{\xi}{2\sqrt{\Delta_T}})}{\tfrac{\xi}{2\sqrt{\Delta_T}}}- e^{-\tfrac{1}{2}\vartheta \xi^2}\Big)^2\frac{d\xi}{\sqrt{\Delta_T}} \\
 \leq &\; IV + V + VI,
\end{align*}
with
\begin{align*}
IV & =  (2\pi)^{-1}\int_{|\xi| \leq \rho \sqrt{\Delta_T}}\Big(e^{-\vartheta\Delta_T\big(1-\cos(\tfrac{\xi}{\sqrt{\Delta_T}})\big)}\frac{\sin(\tfrac{\xi}{2\sqrt{\Delta_T}})}{\tfrac{\xi}{2\sqrt{\Delta_T}}}- e^{-\tfrac{1}{2}\vartheta \xi^2}\Big)^2\frac{d\xi}{\sqrt{\Delta_T}}, \\
V & = 2(2\pi)^{-1} \int_{|\xi| \geq \rho \sqrt{\Delta_T}} e^{-2\vartheta\Delta_T\big(1-\cos(\tfrac{\xi}{\sqrt{\Delta_T}})\big)}\Big(\frac{\sin(\tfrac{\xi}{2\sqrt{\Delta_T}})}{\tfrac{\xi}{2\sqrt{\Delta_T}}}\Big)^2\frac{d\xi}{\sqrt{\Delta_T}}, \\
VI & = 2(2\pi)^{-1}\int_{|\xi| \geq \rho\sqrt{\Delta_T}}e^{-\vartheta \xi^2}\frac{d\xi}{\sqrt{\Delta_T}},
\end{align*}
for any $\rho \geq 0$. By a first order expansion, we have that $IV$ is less than
$$\int_{|\xi |\leq \rho\sqrt{\Delta_T}} e^{-\vartheta \xi^2}\Big(\Big(\frac{\xi^4\alpha(\tfrac{\xi}{\sqrt{\Delta_T}})}{\Delta_T} +\frac{\xi^6\alpha(\tfrac{\xi}{\sqrt{\Delta_T}})}{\Delta_T^2} \Big) e^{\frac{\xi^4\alpha\big(\tfrac{\xi}{\sqrt{\Delta_T}}\big)}{\Delta_T}} +\frac{\xi^2\alpha(\tfrac{\xi}{\sqrt{\Delta_T}})}{\Delta_T}\Big)^2\frac{d\xi}{\sqrt{\Delta_T}}
$$
for some bounded function $\xi \leadsto \alpha(\xi)$. Set $\alpha^\star = \sup_x|\alpha(x)|$. We thus obtain that $IV$ is less than a constant times
\begin{align*}
&  \int_{|\xi|\leq \rho\sqrt{\Delta_T} } \frac{\xi^8}{\Delta_T^2}(\alpha^\star)^2 e^{-(\vartheta-2\rho \alpha^\star) \xi^2}\frac{d\xi}{\sqrt{\Delta_T}}.
\end{align*}
If we pick $\rho$ such that $\vartheta>2\rho\overline{\alpha}$, the term $IV$ is of order $\Delta_T^{-5/2}$. For  the term
 \begin{align*}
 V &=\int_{|\xi|> \rho} e^{-2\vartheta\Delta_T(1-\cos\xi)} \Big(\frac{\sin(\frac{\xi}{2})}{\frac{\xi}{2}}\Big)^{2}d\xi,
 \end{align*}
noting that $(\sin x)^2=\big(1-\cos(2x)\big)/2$, we bound the $2\pi$-periodic, even and continuous function $\xi \leadsto e^{-2\vartheta\Delta_T(1-\cos\xi)} (1-\cos\xi)$ by its supremum $(4e\vartheta\Delta_T)^{-1}$. The integrability of $\xi^{-2}$ away from 0 enables to conclude that $V$ is of order $\Delta_T^{-1}$. Finally, by Gaussian approximation, we readily obtain that $VI$ is of order $\Delta^{-1/2}e^{-\rho^2\vartheta \Delta_T}.$

In conclusion, we have that $\int_{\R} \big(p_{\vartheta,\Delta_T}(x)- q_{\vartheta,\Delta_T}(x)\big)^2dx$ is dominated by the term $V$ and is thus of order
$\Delta_T^{-1}$. It follows that  $I$ is of order $\eta_T^{1/2}\Delta_T^{-1/2}$ and the choice $\eta_T = \kappa\sqrt{\Delta_T \log(T/\Delta_T)}$ implies $I = o\big((T/\Delta_T)^{-1}\big)$ thanks to the restriction condition $T/\Delta_T^{1+\frac{1}{4}}=o((\log( T/\Delta_T))^{-\frac{1}{4}})$. The proof of Theorem \ref{structure statistique macro} is complete.
\subsection{Proof of Theorem \ref{estimateurs}}
Set
$$\xi_{i,T} = \frac{\big(X_{i\Delta_T}-X_{(i-1)\Delta_T}\big)^2-\vartheta \Delta_T}{\big(\lfloor T\Delta_T^{-1}\rfloor \vartheta \Delta_T(1+2\
\vartheta \Delta_T)\big)^{1/2}}.$$
Under $\PP_\vartheta^{T,\Delta_T}$, the variables $\xi_{i,T}$ are independent, identically distributed, and we have
$$\E_\vartheta^{T,\Delta_T}\big[\xi_{i,T}\big]=0\;\;\text{and}\;\;\sum_{i = 1}^{\lfloor T\Delta_T^{-1}\rfloor}\mathrm{Var}\big[\xi_{i,T}\big]=1$$
by \eqref{moment 4}. Moreover, for every $\delta >0$,
$$\sum_{i = 1}^{\lfloor T\Delta_T^{-1}\rfloor} \E_{\vartheta}^{T,\Delta_T}\big[\big|\xi_{i,T}\big|^2{\bf 1}_{\{|\xi_{i,T}| \geq \delta\}}\big] \rightarrow 0\;\;\text{as}\;\;T \rightarrow \infty,$$
therefore, by the central limit theorem
$U_T = \sum_{i = 1}^{\lfloor T\Delta_T^{-1} \rfloor} \xi_{i,T} \rightarrow {\mathcal N}(0,1)$
in distribution under $\PP_\vartheta^{T,\Delta_T}$ as $T\rightarrow \infty$ in all three regimes (microscopic, intermediate and macroscopic). We thus obtain the following representation
$$\widehat \vartheta_{T,\Delta_T}^{QV} = T^{-1}\Delta_T \lfloor T\Delta_T^{-1}\rfloor\vartheta +
T^{-1}\big(\lfloor T\Delta_T^{-1}\rfloor\Delta_T\vartheta(1+2\vartheta \Delta_T)\big)^{1/2}U_T$$
and the result follows from $T^{-1}\Delta_T \lfloor T\Delta_T^{-1}\rfloor \sim 1$ and
$$T^{-2}\lfloor T\Delta_T^{-1}\rfloor\Delta_T\vartheta(1+2\vartheta \Delta_T) \sim I_{T,0}(\vartheta)^{-1}+ I_{T,\infty}(\vartheta)^{-1}$$ as $T \rightarrow \infty$ in all three regimes.

\subsection{Proof of Theorem \ref{cas intermediaire}}

By \eqref{equivalent info} of Lemma \ref{equivalents info}, it suffices to prove
\begin{equation} \label{true deficiency}
\liminf_{T\rightarrow \infty} \mathfrak{I}_{T,\Delta_T}(\vartheta)\big(I_{T,0}(\vartheta)^{-1}+I_{T,\infty}(\vartheta)^{-1}\big) >1.
\end{equation}
Up to taking a subsequence, we may assume that $\Delta_T \rightarrow \Delta \in (0,1/(4\vartheta)]$ as $T \rightarrow \infty$. Using \eqref{technical Fisher} and Theorem \ref{estimateurs}, for every $\vartheta \in \Theta$, we have
\begin{align*}
& \mathfrak{I}_{T,\Delta_T}(\vartheta)\big(I_{T,0}(\vartheta)^{-1}+I_{T,\infty}(\vartheta)^{-1}\big) \\
\sim \;& \Big(\vartheta \Delta \E_{\vartheta}^{T,\Delta}\big[h_{\Delta}(\vartheta, X_{\Delta})^2\big] +2 \E_{\vartheta}^{T,\Delta}\big[|X_{\Delta}| h_{\Delta}(\vartheta, X_{\Delta})\big]+ 1-\vartheta \Delta\Big)(2\vartheta \Delta+1) \\
=:  \;& {\mathcal M}(\vartheta \Delta)\;\;\text{as}\;\;T\rightarrow \infty,
\end{align*}
where ${\mathcal M}$ is a univariate function since
$\PP_\vartheta^{T,\Delta}$ has density $f_{\Delta}(\vartheta, k)= e^{-\vartheta \Delta}{\mathcal I}_{|k|}(\vartheta \Delta)$ with respect to the counting measure on $\Z$. Therefore, Theorem \ref{cas intermediaire} is equivalent to proving that
\begin{equation} \label{minoration M}
\mathcal{M}(x) > 1\;\;\text{for every}\;\;x \in \big(0,\tfrac{1}{4}\big].
\end{equation}
Using (\ref{equadiff}) of Lemma \ref{lemme Bessel} we have
\begin{align*}
& \partial_x \mathcal M(\vartheta\Delta) \\=
 & 1-4\vartheta\Delta+(1+4\vartheta\Delta)\E_{\vartheta}^{T,\Delta}\big[h_{\Delta}(\vartheta,X_{\Delta})^2\big] +  4\E_{\vartheta}^{T,\Delta}\big[|X_{\Delta}| h_{\Delta}(\vartheta,X_{\Delta})\big]  \\
&+2(1+2\vartheta\Delta)\E_{\vartheta}^{T,\Delta}\big[\big(|X_{\Delta}|+\vartheta\Delta h_{\Delta}(\vartheta,X_{\Delta})\big)\partial_\vartheta h_{x}(\vartheta,X_{\Delta})\big]
\end{align*}
where $$\partial_\vartheta h_{\Delta}(\vartheta,k)=\frac{{\mathcal I}_{|k|+2}(\vartheta\Delta)}{{\mathcal I}_{|k|}(\vartheta\Delta)} +\frac{1}{\vartheta \Delta}h_{\Delta}(\vartheta,k)-h_{\Delta}(\vartheta,k)^2$$ is positive (see Theorem 1 of Baricz \cite{Baricz}) and
$h(\vartheta,k)$ is in $[0,1]$ according to (\ref{decroissance}) of Lemma \ref{lemme Bessel}. We derive
$$\partial_x \mathcal{M}(x) \geq 1-4x > 0\;\;\text{for}\;\;x\in\big(0,\tfrac{1}{4}\big],$$
hence \eqref{minoration M}. Since ${\mathcal M}(x) \rightarrow 1$ as $x \rightarrow 0$, the conclusion follows.


\begin{thebibliography}{99}
\small{
\bibitem{Baricz} Baricz, \'A. (2008). Functional inequalities involving Bessel and modified Bessel functions of the first kind. {\it Expositiones Mathematicae} \textbf{26}, 279--293.
\bibitem{Hautsch} Bauwens, L. and Hautsch, N (2006). Modelling high frequency financial data using point processes. {\it Discussion paper}.
\bibitem{Pitts10} Bogsted, M. and Pitts, S. (2010). Decompounding random sums: a nonparametric approach. \textit{Ann Inst Stat Math} \textbf{62}, 855--872.
\bibitem{Buchmann03} Buchmann, B. and Gr\"ubel, R. (2003). Decompounding: an estimation
problem for Poisson random sums. \textit{Ann. Statist} \textbf{31}, 1054--1074.
\bibitem{Buchmann04} Buchmann, B. and Gr\"ubel, R. (2004). Decompounding Poisson random sums: recursively truncated estimates in the discrete case. \textit{Ann. Inst. Math} \textbf{56}, 743--756.
\bibitem{Comte09} Comte, F. and Genon-Catalot, V. (2009). Nonparametric estimation for pure jump L\'evy processes
based on high frequency data. \textit{Stochastic Processes and their
Applications} \textbf{119}, 4088–-4123.
\bibitem{Comte10} Comte, F. and Genon-Catalot, V. (2010). Nonparametric adaptive estimation for pure jump L\'evy processes. \textit{Annales de l'I.H.P., Probability and Statistics} \textbf{46}, 595--617.
\bibitem{IH} Ibragimov, I.A. and Hasminskii, R.Z (1981). {\it Statistical Estimation. Asymptotic Theory}. Springer-Verlag.
\bibitem{CTRW} Masoliver, J., Montero, M., Perell\'o, J. and Weiss, G.H. (2008).
Direct and inverse problems with some generalizations and extensions.
{\it Arxiv preprint 0308017v2}.
\bibitem{LCY} Le Cam, L. and Yang, L.G. (2000) {\it Asymptotics in Statistics: Some Basic Concepts}. 2nd edition. New
York: Springer-Verlag.
\bibitem{Nasell} Nasell, I. (1974). Inequalities for Modified Bessel Functions. {\it Math. Comp.} {\bf 28}, 253 -- 256.
\bibitem{Reiss} Neumann, M. and Rei{\ss}, M. (2009). Nonparametric estimation for L\'evy processes from low-frequency observations. \textit{Bernoulli} \textbf{15}, 223--248.
\bibitem{Sato} Sato, K-I. (1999). {\it L\'evy Processes and Infinitely Divisible Distributions}. Cambridge University Press.
\bibitem{VdV} van der Vaart, A.W. (1998). {\it Asymptotic Statistics}. Cambridge University Press.
\bibitem{VAN-ES07} van Es, B., Gugushvili, S. and Spreij, P. (2007). A kernel type nonparametric density estimator for decompounding. \textit{Bernoulli} \textbf{13}, 672--694.
\bibitem{Watson} Watson, G.N (1922). {\it A Treatise on the Theory of Bessel Functions}. Cambridge University Press.
}
\end{thebibliography}
\end{document}